\documentclass{amsart}
\usepackage{amssymb}
\usepackage{amsmath}
\numberwithin{equation}{section}
\usepackage{amsthm}
\usepackage{appendix}
\usepackage{bookmark}
\usepackage{setspace}
\usepackage{mathrsfs}
\usepackage{mathtools}
\usepackage{tikz-cd}
\usepackage{manfnt}
\usepackage{fullpage}
\newtheorem{thm}{Theorem}
\numberwithin{thm}{section}
\newtheorem{prop}[thm]{Proposition}
\newtheorem{lem}[thm]{Lemma}
\newtheorem{cor}[thm]{Corollary}
\theoremstyle{remark}
\newtheorem{rem}[thm]{Remark}
\theoremstyle{definition}

\newcommand{\1}{\mathbf{1}}

\newcommand{\C}{\mathbf{C}}
\newcommand{\cali}[1]{\mathcal{#1}}
\newcommand{\chx}[1]{\langle #1\rangle}
\newcommand{\cl}{\mathrm{Cl}}

\newcommand{\comment}[1]{}

\newcommand{\gal}{\mathrm{Gal}}

\newcommand{\got}[1]{\mathfrak{#1}}

\newcommand{\Hom}{\mathrm{Hom}}

\newcommand{\ilim}{\varinjlim}
\newcommand{\im}{\mathrm{im}}

\newcommand{\pair}[2]{\langle #1, #2 \rangle}

\newcommand{\plim}{\varprojlim}

\newcommand{\Q}{\mathbf{Q}}

\newcommand{\R}{\mathbf{R}}
\newcommand{\re}{\mathrm{Re}}
{\normalsize {\normalsize {\normalsize {\normalsize {\tiny }}}}}

\newcommand{\Scr}[1]{\mathscr{#1}}

\newcommand{\Z}{\mathbf{Z}}

\newcommand{\mes}{\mathrm{Mes}}
\newcommand{\nm}{\mathrm{Nm}}

\title{Sum Expressions for $p$-adic Hecke $L$-functions of Totally Real Fields}

\author{Luochen Zhao}

\date{Jan 4, 2023}



\address{Department of Mathematics, Johns Hopkins University, 404 Krieger Hall, 3400 N.~Charles Street, Baltimore, MD 21218, USA}
\email{lzhao39@jhu.edu}
\begin{document}
	\begin{abstract}
		As a continuation of previous work, we establish sum expressions for $p$-adic Hecke $L$-functions of totally real fields in the sense of Delbourgo, assuming a totally real analog of Heegner hypothesis. This is done by finding explicit formulas of the periods of the corresponding $p$-adic measures. As an application, we extend the Ferrero-Greenberg formula of derivatives of $p$-adic $L$-functions to this setting.
	\end{abstract}
	
	\maketitle
	
	\setcounter{tocdepth}{1}
	\tableofcontents
	
	\section{Introduction}
	\label{section.intro}
	
	Let $p$ be a rational prime, $F/\Q$ be a totally real field, $\cali{O}$ be its ring of integers, $\nm:F\to \Q$ be the norm map, and $\chi$ be a Hecke character of the narrow ray class group $\cl_+(\cali{N})$ for some nonzero integral ideal $\cali{N}$. The complex Hecke $L$-function attached to $\chi$ is given by (when $\re(s)>1$)
	\begin{align*}
		L_F(s,\chi) = \sum_{0\ne \got{a}\subseteq \cali{O}} \frac{\chi(\got{a})}{\nm(\got{a})^s}.
	\end{align*}
	By the work of Deligne-Ribet \cite{DR80}, Cassou-Nogu\`es \cite{CN79} and Barsky \cite{Ba78}, there exists a $p$-adic Hecke $L$-function $L_{F,p}(s,\chi)$ on $\Z_p$ that interpolates its complex counterpart for all $m\in \Z_{>0}$ (see \cite[p.~105]{Hida}, \cite[Theorem 8.2]{DR80}):
	\begin{align}
		\label{equation.interpolation}
		L_{F,p}(1-m,\chi) = \prod_{\got{p}\mid p}(1-\chi\omega_F^{-m}(\got{p})\nm(\got{p})^{m-1}) L_F(1-m,\chi\omega_F^{-m}),
	\end{align}
	where $\omega_F$ is the composition of $\nm$ and the Teichm\"uller character $\omega$, and the right hand side is algebraic by Siegel and Klingen. In this article, we will establish a sum expression for $L_{F,p}(s,\chi)$ in the sense of Delbourgo, by assuming what we call the \textit{Cassou-Nogu\`es condition} for $\cali{N}$, an analogue of Heegner hypothesis in the totally real case. In fact the general case with an auxiliary Euler factor will also follow via a simple adaptation of our proof. We note that such expressions are known to exist for Kubota-Leopoldt $p$-adic $L$-functions by works of Delbourgo \cite{De06,De09,De09-2}, Knospe-Washington \cite{KW21} and the author \cite{Zh22}. For a detailed discussion, we refer the reader to the introduction of \cite{Zh22}.
	
	\subsection{Review of Shintani's method}
	\label{subsection.intro.review-shintani}
	
	To state our results properly, we recall Shintani's treatment of Hecke $L$-functions, along the way fixing some notation used throughout this article. Let $F$ and $\chi$ be as above, and let $k=[F:\Q]$. Fix a numbering of real places of $F$, $\{\sigma_1,\cdots,\sigma_k\}$, so that we have an embedding $F\hookrightarrow \R^k\simeq F\otimes_{\Q} \R$. Denote by $\R_+$ the positive real numbers. For any subset $X$ of $F$, write $X_+ = X\cap \R_+^n$, and denote $\cali{O}^\times_+$ particularly by $E$. Fix the direction vector $e_k=(0,0,\cdots,0,1)\in \R^n$. From the work of Shintani \cite{Sh76}, Colmez \cite{Co88} (see also \cite{DF12} for a detailed discussion of Colmez's construction in the cubic case) and Yamamoto \cite{Ya10}, there exists a finite collection $\{V\}$, each $V=\{v_1,\cdots,v_k\}$ being a $\Q$-basis of $F$ and a subset of $\cali{O}_+$ such that the following cone decomposition holds \cite[Proposition 5.6]{Ya10}:
	\begin{align}\label{eq:shintani-decomp}
		(F\otimes \R)_+ = \R_+^n = \bigsqcup_{\varepsilon\in E}\bigsqcup_{V} \varepsilon \overline{C}(V).
	\end{align}
	Here $C(V) = \sum_{1\le i\le k} \R_+ v_i$, and $\overline{C}(V)$ is the upper closure of $C(V)$ with respect to $e_k$ \cite[Definition 5.5]{Ya10}. Let $P(V)$ be the fundamental parallelotope of $\overline{C}(V)$, that is, the subset of $\overline{C}(V)$ determined by the property that any $v\in \overline{C}(V)$ can be uniquely written in the form $v= v_0 + \sum_{1\le i\le k} n_i v_i$ for some $v_0\in P(V)$ and $(n_i)_{1\le i\le k}\in \Z_{\ge 0}^k$. Take a set of prime-to-$\cali{N}$ integral ideal representatives of the narrow class group $\cl_+(1)$, $\{\got{a}_1,\cdots,\got{a}_h\}$. Then as observed by Shintani, since $(\got{a}_{i}^{-1})_+/E$ is in set bijection with
	\begin{align*}
		\bigsqcup_V \got{a}_i^{-1}\cap \overline{C}(V) = \bigsqcup_V \bigsqcup_{x\in P(V)\cap \got{a}_i^{-1}} x+\Z_{\ge 0} v_1+\cdots+\Z_{\ge 0}v_k,
	\end{align*}
	we have
	\begin{align*}
		L_F(s,\chi) = \sum_{1\le i\le h} \frac{\chi(\got{a}_i)}{\nm(\got{a}_i)^s}
		\sum_V \sum_{x\in P(V)\cap \got{a}_i^{-1}} \sum_{n_1,\cdots,n_k\ge 0}
		\frac{\chi(x+n_1v_1+\cdots+n_kv_k)}{\nm(x+n_1v_1+\cdots+n_kv_k)^s}.
	\end{align*}
	Throughout this article, we will fix a Shintani cone decomposition.
	
	\subsection{Main results}
	\label{subsection.intro.main-results}
	
	We introduce now the presiding assumptions:
	\begin{enumerate}
		\item[(A1)] We will always suppose that $\cali{N}$ satisfies the Cassou-Nogu\`es condition
		\begin{align*}
			\cali{O}/\cali{N} \simeq \Z/N \quad\text{for }N=\nm(\cali{N});
		\end{align*}
		often we will simply say that $\cali{N}$ is a Cassou-Nogu\`es ideal. Additionally, we always suppose that $\cali{N}\ne\cali{O}$, and is prime to $p$.
		
		\item[(A2)] We choose each fractional ideal $\got{a}_i$ to be integral and prime to both $p$ and $\cali{N}$.
		
		\item[(A3)] For all $V$ in the Shintani cone decomposition and each $v_i\in V$, we require that $v_i\in\cali{O}_+$. Moreover, in order to apply Euler's method to remove the pole, we assume that $v_i$ is prime to $\cali{N}$ as in \cite[p.~38, (6.ii)]{CN79} and \cite[top of p.~41]{Ka81}.
	\end{enumerate}
	
	The following result is the culmination of the computations done in §\ref{section.period1}-§\ref{section.period2}, whose proof forms the trunk of this paper.
	\begin{thm}
		\label{thm.main}
		Let the notation and assumptions be as above. Let $q>1$ be a power of $p$ that is congruent to $1\bmod N$. Furthermore, let $\psi$ be a character on $\cl_+(p^\infty)=\plim_n \cl_+(p^n)$. We have
		\begin{align*}
			&(1-\psi(\cali{N})\chx{N}^{1-s})L_{F,p}(s,\psi) \\
			=& \lim_{n\to\infty} \sum_{1\le i\le h} \frac{\psi\omega_F^{-1}(\got{a}_i)}{\chx{\nm(\got{a}_i)}^s}\sum_{V} \sum_{x\in P(V)\cap \got{a}_i^{-1}} \sum_{\substack{0\le l_1,\cdots,l_k< q^n\\ \gcd(p,x+\sum_{1\le i\le k}l_iv_i) = 1}} a_{V,\cali{N}}(x+\sum_{1\le i\le k}l_iv_i)
			\frac{\psi\omega_F^{-1}(x+\sum_{1\le i\le k}l_iv_i)}{\chx{\nm(x+\sum_{1\le i\le k}l_iv_i)}^s},
		\end{align*}
		where the coefficients are given by
		\begin{align*}
			a_{V,\cali{N}}(y) = \frac{(-1)^{k-1}}{N^{k-1}}\sum_{\substack{1\le d_1,d_2,\cdots,d_k<N\\ d_1v_1+\cdots+d_kv_k\equiv -y \bmod \cali{N}}}d_1d_2\cdots d_k.
		\end{align*}
		When $\chi$ is a Hecke character on $\cl_+(\cali{N})$ with nontrivial narrow modulus, i.e., $\chi$ is not from $\cl_+(1)$, we have
		\begin{align*}
			&L_{F,p}(s,\chi\psi)\\
			=& \lim_{n\to\infty} \sum_{1\le i\le h} \frac{\chi\psi\omega_F^{-1}(\got{a}_i)}{\chx{\nm(\got{a}_i)}^s}\sum_{V} \sum_{x\in P(V)\cap \got{a}_i^{-1}} \sum_{\substack{0\le l_1,\cdots,l_k< q^n\\ \gcd(p,x+\sum_{1\le i\le k}l_iv_i) = 1}} 
			a_{V,\chi}(x+\sum_{1\le i\le k}l_iv_i)\frac{\psi\omega_F^{-1}(x+\sum_{1\le i\le k}l_iv_i)}{\chx{\nm(x+\sum_{1\le i\le k}l_iv_i)}^s},
		\end{align*}
		where 
		\begin{align*}
			a_{V,\chi}(x+\sum_{1\le i\le k}l_iv_i) = (-1)^k \sum_{0\le d_i<l_i,1\le i\le k} \chi(x+\sum_{1\le i\le k} d_iv_i).
		\end{align*}
	\end{thm}

	\begin{rem}
		Actually, following \cite{CN79}, we can establish a sum expression when the conductor of $\chi$ is arbitrary, at the expense of multiplying $L_{F,p}(s,\chi\psi)$ by an auxiliary Euler factor as in the first part of the theorem; see Remark \ref{rem:general-sumexpr}. Still, the Cassou-Nogu\`es condition seems indispensable if we want to remove the extra Euler factor.
	\end{rem}
	
	\begin{rem}
		We indicate conceptually how Theorem \ref{thm.main} is proved. As discussed in \cite[§1.5]{Zh22}, the sum expressions can be established from three ingredients: integral representations, computability of periods, and the uniform periodicity of the periods. In our setting, the first ingredient is available thanks to the work of Cassou-Nogu\`es, complemented by a reinterpretation due to Katz \cite{Ka81}. As such, the main novelty of the present proof is the determination of the explicit formulas (Theorem \ref{thm.first-period-formula} and \ref{thm.second-period-formula}) of the periods concerned, whereby uniform periodicity automatically follows.
	\end{rem}
	
	\begin{rem}
		An interesting consequence is that the periods mentioned above are $\Z[1/N,\im(\chi)]$-valued, which previously are known to be valued in a finite extension of $\Z_p$. As such, this allows us to speak of their $2$-divisibilities even when $p$ is odd. In Appendix \ref{section.appendix-b}, we shall compute some approximate coefficients of the Iwasawa functions attached to the $p$-adic $L$-functions of $\Q(\sqrt{5})$ where $p$ is odd, and remarkably the Deligne-Ribet $2$-divisibility (see, e.g., \cite[(4.8)]{Ri79}), \textit{a priori} only making sense when $p=2$, propagates in all the numerical examples considered.
	\end{rem}
	
	\begin{rem}
		It will also follow from our proof that one can use these expressions to compute $p$-adic $L$-values with a precision of $O(q^n)$, by taking the finite sums on the right hand sides of both expressions without limits. Regarding this, it is tempting to compare this method to those of \cite{Ro15} and Lauder-Vonk \cite{LV22}.
	\end{rem}
	
	As a byproduct, following the sum expression-to-derivative philosophy demonstrated in \cite[§4]{Zh22}, we obtain
	\begin{cor}[Generalized Ferrero-Greenberg formula]
		\label{cor.ferrero-greenberg-formula}
		Assume in addition that
		\begin{enumerate}
			\item[(A4)] The prime $p$ is inert in $F$ and $\chi(p)=1$.
			
			\item[(A5)] For all $V$ in the cone decomposition, $V$ is a $\Z_p$ basis of $\cali{O}_p=\plim_n\cali{O}/p^n$.
		\end{enumerate}
		Then, we have
		\begin{align*}
			L_{F,p}'(0,\chi\omega_F)
			= (-1)^{k-1}\sum_{1\le i\le h} \chi(\got{a}_i) \sum_V \sum_{x\in P(V)\cap \got{a}_i^{-1}} \sum_{0\le d_1,\cdots,d_k<N} \chi(x + \sum_{1\le i\le k} d_iv_i)\log_p\Gamma_{F,p,V}\left(\frac{x+\sum_{1\le i\le k}d_iv_i}{N}\right).
		\end{align*}
	\end{cor}
	
	Here, the multiple $p$-adic Gamma function $\Gamma_{F,p,V}$ is defined in \eqref{definition.Gamma}.
	
	\begin{rem}
		\label{rem.O_Vp}
		For a given $V$, as long as $p\nmid [\cali{O}:\sum_{1\le i\le k} \Z v_i]$, $\cali{O}_p= \sum_{1\le i\le k}\Z_p v_i$. 
	\end{rem}
	
	\begin{rem}
		It is worth pointing out that formulas of a similar guise can be found in the work of Cassou-Nogu\`es \cite[Th\'eor\`em 6]{CN79b} and that of Kashio \cite[Theorem 6.2]{Kas05}. However, the above formula appears to be of a distinct nature, as our multivariate $p$-adic Gamma function is defined elementarily in the spirit of Morita \cite{Mo75}, while the Gamma functions in aforementioned papers are constructed inexplicitly as certain derivatives.
	\end{rem}
	
	\begin{rem}
		Using the sum expression as sketched in Remark \ref{rem:general-sumexpr}, we can derive a Ferrero-Greenberg type formula for an arbitrary conductor which involves an auxiliary Cassou-Nogu\`es ideal. The details will be discussed in a separate paper.
	\end{rem}

	The assumption that $p$ is inert and $\chi(p)=1$ fits the special case of the Gross-Stark conjecture, now a theorem of Dasgupta-Darmon-Pollack \cite{DDP} and Dasgupta-Kakde-Ventullo \cite{DKV}, when the vanishing order of the $L$-function is one. See \cite[§3]{Gr81} or a summary in \cite[§2.1]{Da08}. Combining this and the corollary, we get the following Gross-Koblitz type formula (\textit{cf.~}\cite[Theorem 1.7]{GK79}):
	
	\begin{cor}
		Keep the assumptions in Corollary \ref{cor.ferrero-greenberg-formula}.
		Let $H/F$ be the narrow ray class field attached to $\cl_+(\cali{N})/(p)^{\Z}$, $U_p^-=\{u\in H^\times: |u|_v = 1\text{ for all place }v\nmid p\}$, and assume
		\begin{itemize}
			\item[(A6)] $H$ is a CM extension.
		\end{itemize}
		Let $u_p\in \Q\otimes_\Z U_p^-$ be the rational Brumer-Stark unit (denoted by $u=u(\got{P})$ in \cite[Conjecture 3.13]{Gr81}, where $\got{P}$ is a fixed prime above $p$ in $H$). For any fractional ideal $\got{a}$ prime to $\cali{N}$, let $\sigma_{\got{a}}\in \gal(H/F)$ be the image of $\got{a}$ under Artin reciprocity, and write $\got{a} = \got{a}_i(y)$ for some unique $1\le i\le h$ and $y\in \overline{C}(V)$, where $V$ appears in the Shintani cone decomposition and is uniquely determined by $\got{a}$. Denote also by $\nu$ the order of $(p)$ in $\cl_+(\cali{N})$. Then
		\begin{align}
			\label{equation.gross-koblitz}
			\log_p \nm_{F\otimes_{\Q} \Q_p/\Q_p}(u_p^{\sigma_{\got{a}}}) = (-1)^k\sum_{0\le j<\nu} \sum_{x\in P(V)\cap \got{a}_i^{-1}} \sum_{\substack{0\le d_1,\cdots,d_k<N\\ x+\sum_{1\le i\le k}d_i v_i \equiv p^jy\pmod{\cali{N}}}} \log_p\Gamma_{F,p,V}\left(\frac{x+\sum_{1\le i\le k}d_iv_i}{N}\right).
		\end{align}
	\end{cor}
	\begin{rem}
		The assumption (A6), as is already present in \cite{Gr81}, ensures that there are totally odd characters on $\gal(H/F)$; otherwise all the $p$-adic $L$-functions $L_{F,p}(s,\chi\omega_F)$ attached to characters $\chi$ of $\gal(H/F)$ are identically zero.
	\end{rem}
	
	For a direct $p$-adic analytic formula of $u_p$ via the multiplicative integral, we refer the reader to works of Dasgupta and collaborators, for example \cite[Proposition 3.3]{Da08}; for another formula of $u_p$ in terms of the Dedekind-Rademacher cocycle when $F/\Q$ is quadratic, see \cite{DPV,DPV2}. We also draw attention to the recent breakthroughs of Dasgupta-Kakde \cite{DK20,DK21}, which establish the integrality of $u_p$ away from 2, as well as the multiplicative integral representation under certain assumptions.

	\subsection{Outlook}
	\label{subsection.intro.outlook}
	
	In the spirit of Iwasawa \cite{Iw58}, Ferrero \cite{Fe78} and Ferrero-Washington \cite{FW79}, the explicit period formulas \eqref{equation.period-zeta} and \eqref{equation.second-period-formula} are expected to play important roles in understanding the analytic $\mu$- and $\lambda$-invariants of abelian extensions of $F$; see Appendix \ref{section.appendix-b} for more detail. If we further assume that $p\ne 2$ and the degree of the extension is prime to $p$, then results of Wiles \cite[Theorem 1.3 and Theorem 1.4]{Wi90} assert that the analytic and algebraic Iwasawa invariants coincide. We hope to investigate this question in the future.
	
	\vspace{3mm}
	
	In another direction, it would be very desirable to know if equation \eqref{equation.gross-koblitz} could shed light on the explicit construction of the Brumer-Stark unit $u_p$, which, when $F=\Q$, is known to be essentially a Gauss sum \cite{GK79}. In fact, to the best of the author's knowledge, it is not clear whether a complex analogue of \eqref{equation.gross-koblitz} exists, unless $F=\Q$, in which case it is a result of Deligne \cite{De82}.

	\subsection{Notation}
	
	\label{subsection.intro.notation}
	
	We will retain the notation introduced above; this includes our assumptions made at the beginning of §\ref{subsection.intro.main-results}. In the rest of this paper, $\chi$ exclusively denotes a finite Hecke character on $\cl_+(\cali{N})$ of nontrivial narrow modulus. For any $h\in\Z_{>1}$ and $a\in \Z/h$, we denote by $a^\flat_h$ and $a^\sharp_h$ the unique integers in $[0,h)$ and $(0,h]$ respectively, such that $a\equiv a^\flat_h\equiv a^\sharp_h\bmod h$. Additionally, thanks to the Cassou-Nogu\`es condition, we write $a^\flat_\cali{N}$ for $(a\bmod \cali{N})^\flat_{N}$ if $a\in F$ is $\cali{N}$-integral, and similarly for $a^\sharp_\cali{N}$. Also, denote by $\cali{O}_p$ the $p$-adic ring $\cali{O}\otimes_{\Z}\Z_p$ and $F_p$ the algebra $F\otimes_{\Q} \Q_p$. Throughout we will fix embeddings of $\bar{\Q}$ into $\C$ and $\C_p$, so we can regard $\bar{\Q}$ as a subfield of both. The letter $x$ is reserved to denote an element of $F$ that is $\cali{N}p$-integral, and $V$ is reserved to denote a $\Q$-basis of $F$, its elements being $v_1,\cdots,v_k$, all of which are in $\cali{O}_+$ and are prime to $\cali{N}$. A governing convention is the vectorial notation: frequently we abbreviate a tuple $(a_1,\cdots,a_k)$ as simply $a$, so we have $a+b = (a_1+b_1,\cdots,a_k+b_k)$ and $a\cdot b = \sum_{1\le i\le k}a_i b_i$. Moreover, if $a$ and $b$ are two tuples, we understand $a\le b$ as inequalities for all components, and the same for $a<b$, etc.; such use will be propagated when $a\in \R^k$ and $b\in \R$, in which case $b$ is understood as the vector $(b,b,\cdots,b)$. If a generator set $V=\{v_1,\cdots,v_k\}$ as above is fixed, for any $z\in F_p$, we always suppose $z$ is of the form $\sum_{1\le i\le k} z_i v_i$ with $z_i\in \Q_p$; this is legitimate since $V$ is also a $\Q_p$-basis of $F_p$. Also put $\cali{O}_{V} =\Z v_1+\cdots+\Z v_k$ and $\cali{O}_{V,p}=\Z_p v_1+\cdots+\Z_p v_k$, for which the previous convention applies by regarding both $\cali{O}_V$ and $\cali{O}_{V,p}$ as subgroups of $F_p$. Finally, we denote by $\1_A$ the indicator function that has value 1 if the condition $A$ is true and otherwise 0, and, given a finite group $G$, denote by $G^\wedge$ the dual group $\Hom(G,\bar{\Q}^\times)$. 
	
	\subsection{Acknowledgement}
	
	The author is greatly indebted to Antonio Lei, for suggesting this line of works to us, for having many insightful discussions, and for reading the many drafts that eventually led to this paper. He is also grateful to Jan Vonk for helpful comments. Finally he thanks the referee for valuable feedback that enhances the presentation of this article.
	
	\section{Preliminaries on $p$-adic measures}
	\label{section.preliminaries}
	
	We give here a quick recapitulation of some background material; more detailed accounts can be found in \cite[§3.6-§3.9]{Hida} and \cite{Ka81}. Recall $\cl_+(p^\infty)$ denotes the completion $\plim_n\cl_+(p^n)$. Following Cassou-Nogu\`es, there exists a $p$-adic measure $\mu_{F,\cali{N}}$ on $\cl_+(p^\infty)$ such that for any Hecke character $\psi$ on $\cl_+(p^\infty)$,
	\begin{align}
		\label{equation.integral-representation-zeta}
		-(1-\psi(\cali{N})\chx{N}^{1-s})L_{F,p}(s,\psi) = \int_{\cl_+(p^\infty)} \psi\omega_F^{-1}(\alpha) \chx{\nm\alpha}^{-s} \mu_{F,\cali{N}}(\alpha).
	\end{align}
	
	When the Hecke character $\chi$ is present, there is a measure $\mu_{F,\chi}$ on $\cl_+(p^\infty)$ such that for all $\psi$,
	\begin{align}
		\label{equation.integral-representation-dirichlet}
		L_{F,p}(s,\chi\psi) = \int_{\cl_+(p^\infty)} \psi\omega_F^{-1}(\alpha)\chx{\nm\alpha}^{-s}\mu_{F,\chi}(\alpha);
	\end{align}
	the removal of the auxiliary Euler factor reflects the regularity of $L_{F,p}(s,\chi\psi)$ at $s=1$. Essentially, the construction of these $p$-adic measures can be summarized in two steps:
	\begin{itemize}
		\item[(i)] Given a tuple $(V,x)$ where $V$ is a generator set and $x\in F$ is $\cali{N}p$-integral, let $t\in \Z_{\ge 0}$ be such that $x\in p^{-t}\cali{O}_{V,p}\subset F_p$ and denote by $V^\dagger=p^{-t}V=\{v_1'= p^{-t}v_1,\cdots,v_k'=p^{-t}v_k\}$ and $\cali{O}_{V^\dagger,p} = p^{-t}\cali{O}_{V,p}$. Then one may construct certain $p$-adic measures $\mu_{V,x,\cali{N}}$ and $\mu_{V,x,\chi}$ on $\cali{O}_{V^\dagger,p}$ that are supported on $x+\cali{O}_{V,p}\subseteq \cali{O}_p$.
		
		\item[(ii)] Let $i_\infty:\cali{O}_p^\times \to \cl_+(p^\infty)$ be the canonical map (see \cite[p.~103, (1b)]{Hida}), and let $g_{i,V}:\cali{O}_{V^\dagger,p}\cap\cali{O}_p^\times \hookrightarrow \cali{O}_p^\times \xrightarrow{i_\infty} \cl_+(p^\infty)\xrightarrow{\got{a}\mapsto \got{a}_i\got{a}} \cl_+(p^\infty)$ be the composition and $(g_{i,V})_*$ the induced pushforward on measures. The constructions in (i) are assembled to form:
		\begin{align}
			\label{equation.push1}
			\mu_{F,\cali{N}} = \sum_{1\le i\le h}\sum_V \sum_{x\in P(V)\cap \got{a}_i^{-1}} (g_{i,V})_*(\mu_{V,x,\cali{N}}|_{\cali{O}_{V^\dagger,p}\cap\cali{O}_p^\times}),
		\end{align}
		
		\begin{align}
			\label{equation.push2}
			\mu_{F,\chi} = \sum_{1\le i\le h}\sum_V\sum_{x\in P(V)\cap \got{a}_i^{-1}} \chi(\got{a}_i)(g_{i,V})_*(\mu_{V,x,\chi}|_{\cali{O}_{V^\dagger,p}\cap\cali{O}_p^\times}).
		\end{align}
	\end{itemize}
	
	We elaborate slightly on the first step. Let $(V,x)$ be given and let $R$ be a finite flat extension of $\Z_p$. By restricting the Amice transform, essentially Cartier duality \`a la Katz \cite[Theorem 1]{Ka81}, $R$-valued measures on $\cali{O}_{V^\dagger,p}$ are in bijection with elements in the formal algebra
	\begin{align*}
		\cali{A}_{V^\dagger} \simeq \plim_n R[t_1,\cdots,t_k]/(t_1^{p^n}-1,\cdots,t_k^{p^n}-1) = R[[t_1-1,\cdots,t_k-1]].
	\end{align*}
	
	Formally we denote the isomorphism by $\Scr{A}: {\rm Mes}(\cali{O}_{V^\dagger,p},R) \xrightarrow{\sim} \cali{A}_{V^\dagger}$, where ${\rm Mes}(\cali{O}_{V^\dagger,p},R)$ stands for the set of measures on $\cali{O}_{V^\dagger,p}$ valued in $R$. In turn, the constructions of $\mu_{V,x,\cali{N}}$ and $\mu_{V,x,\chi}$ can be achieved by manufacturing certain power series $f_{V,x,\cali{N}}$ and $f_{V,x,\chi}$, respectively. We postpone the minutiae of these power series, in fact rational functions, to individual sections below. For now, we are content to record a general formula to be used for period computations. For ease of notation, given $\alpha = \alpha'_1v'_1+\cdots+\alpha'_kv'_k\in \cali{O}_{V^\dagger,p}$, denote by $t^\alpha$ the monomial $t_1^{\alpha'_1}t_2^{\alpha'_2}\cdots t_k^{\alpha'_k}\in \cali{A}_{V^\dagger}$. For any $\varphi\in (\cali{O}_{V^\dagger,p}/p^n)^\wedge$, we define the evaluation $\cdot|_{\varphi}$ on $\cali{A}_{V^\dagger}$ by dictating $t^\alpha|_{\varphi}= \varphi(\alpha)$ for all $\alpha\in\cali{O}_{V^\dagger,p}$.
	
	\begin{prop}
		\label{prop.period-formula-setup}
		Let $\mu\in {\rm Mes}(\cali{O}_{V^\dagger,p},R)$. For all $a\in \cali{O}_{V^\dagger,p}$ and $n\in \Z_{\ge 0}$, we have
		\begin{align*}
			\mu(a+p^n\cali{O}_{V,p}) = \frac{1}{p^{(n+t)k}}\sum_{\varphi\in(\cali{O}_{V^\dagger,p}/p^{n+t})^\wedge} \varphi^{-1}(a)\Scr{A}_{\mu}|_{\varphi}.
		\end{align*}
	\end{prop}
	\begin{proof}
		The $p$-adic module $\cali{O}_{V^\dagger,p}$ comes equipped with a basis $\cali{O}_{V^\dagger,p}\simeq \oplus_{1\le i\le k} \Z_p \cdot v_i'$, and the Amice transform for the formal torus of $\cali{A}_{V^\dagger}$ respects this splitting \cite[proof of Theorem 3.7.1]{Hida}, i.e., the isomorphism \textit{loc.~cit.}~is given by taking the completion of the tensor of the one-dimensional isomorphisms $R[[t_i-1]]\simeq \mes(\Z_p\cdot v_i',R)$. The formula then follows from that in the one-dimensional case, which can be found in p.~84, \textit{ibid}. (Note here that $p^n\cali{O}_{V,p} = p^{n+t}\cali{O}_{V^\dagger,p}$.)
	\end{proof}
	
	\begin{rem}
		As pointed out by Remark \ref{rem.O_Vp}, the passage from $V$ to $V^\dagger$ is unnecessary for all but finitely many primes $p$.
	\end{rem}
	
	\section{Explicit period formula: the zeta case}
	\label{section.period1}
	
	Given $x\in F$ that is $\cali{N}p$-integral and $V=\{v_1,\cdots,v_k\}$, the power series that corresponds to $\mu_{V,x,\cali{N}}$ is given by \cite[§3.8]{Hida}
	\begin{align*}
		f_{V,x,\cali{N}}(t) = \sum_{\xi\ne 1\in (\cali{O}/\cali{N})^\wedge} \frac{\xi(x)t^x}{\prod_{1\le i\le k}(1-\xi(v_i) t^{v_i})} \in \cali{A}_{V^\dagger},
	\end{align*}
	and we are interested in computing
	\begin{align}
		\label{equation.period-define}
		\mu_{V,x,\cali{N}}(a+p^n\cali{O}_{V,p}) = \frac{1}{p^{(n+t)k}}\sum_{\varphi\in(\cali{O}_{V^\dagger,p}/p^{n+t})^\wedge} \varphi^{-1}(a)f_{V,x,\cali{N}}|_\varphi.
	\end{align}
	For this purpose, we break $f_{V,x,\cali{N}}$ up as the sum $\sum_{\xi\ne 1\in (\cali{O}/\cali{N})^\wedge} f_{V,x,\cali{N},\xi}$, where $f_{V,x,\cali{N},\xi} = \frac{\xi(x)t^x}{\prod_{1\le i\le k}(1-\xi(v_i)t^{v_i})}$, and we compute periods of corresponding measures $\mu_{V,x,\cali{N},\xi}$ individually.
	
	\subsection{First step}
	
	Set
	\begin{align*}
		R_{V,x}(a,p^n) =\left\{y= x+\sum_{1\le i\le k}l_iv_i\in F: l_i\in\Z, 0\le l_i<p^n, y-a \in p^n\cali{O}_{V,p}\right\}.
	\end{align*}
	We prove
	\begin{lem} 
		\label{lem.period-step1}
		Suppose $\xi\in (\cali{O}/\cali{N})^\wedge$ is nontrivial. Then 
		\begin{align}
			\label{equation.period-step1}
			\mu_{V,x,\cali{N},\xi}(a+p^n\cali{O}_{V,p}) = \sum_{y\in R_{V,x}(a,p^n)}\frac{\xi(y)}{\prod_{1\le i\le k}(1-\xi(p^n v_i))}.
		\end{align}
	\end{lem}
	
	\begin{proof}
		Using Proposition \ref{prop.period-formula-setup} it suffices to prove the equality in $\bar{\Q}$, thus in $\C$. Introduce an auxiliary real parameter $0<u<1$, and for $y\in F$, write $u^y = u^{\sigma_1(y)}$ with $\sigma_1$ from §\ref{subsection.intro.review-shintani}. Then we have
		\begin{align*}
			\frac{1}{p^{(n+t)k}}\sum_{\varphi\in(\cali{O}_{V^\dagger,p}/p^{n+t})^\wedge} \varphi^{-1}(a) f_{V,x,\cali{N},\xi}|_\varphi = \lim_{u\to 1^-} \frac{1}{p^{(n+t)k}}\sum_{\varphi\in(\cali{O}_{V^\dagger,p}/p^{n+t})^\wedge} \frac{u^x\xi(x)\varphi(x-a)}{\prod_{1\le i\le k}(1-u^{v_i}\xi(v_i)\varphi(v_i))}.
		\end{align*}
		
		Next, identify $(\Z/p^{n+t})^k$ with $\Hom(\cali{O}_{V^\dagger,p}/p^{n+t}, \Z/p^{n+t})$ by the pairing $\pair{z}{w} = \sum_{1\le i\le k}z'_iw_i$, where $z=\sum_{1\le i\le k}z'_iv'_i\in \cali{O}_{V^\dagger,p}/p^{n+t}$ and $w\in (\Z/p^{n+t})^k$. Choose a primitive $p^{n+t}$-th root of unity $\zeta$ and further identity $(\Z/p^{n+t})^k$ with $\Hom(\cali{O}_{V^\dagger,p}/p^{n+t}, \mu_{p^{n+t}})$, so $w(z) = \zeta^{\chx{w,z}}$ for $w,z$ as before. We then find
		\begin{align*}
			\frac{1}{p^{(n+t)k}}\sum_{\varphi\in(\cali{O}_{V^\dagger,p}/p^{n+t})^\wedge} \frac{u^x\xi(x)\varphi(x-a)}{\prod_{1\le i\le k}(1-u^{v_i}\xi(v_i)\varphi(v_i))}
			&= \frac{1}{p^{(n+t)k}} \sum_{0\le w_1,\cdots,w_k<p^{n+t}}
			\frac{u^x\xi(x)\zeta^{\pair{x-a}{w}}}{\prod_{1\le i\le k} (1-u^{v_i}\xi(v_i)\zeta^{\chx{v_i,w}})}\\
			&= \frac{1}{p^{(n+t)k}} \sum_{0\le w_1,\cdots,w_k<p^{n+t}}
			\sum_{l_1,\cdots,l_k\ge 0} u^{x+l\cdot v} \xi(x+l\cdot v) \zeta^{\pair{x-a+l\cdot v}{w}}\\
			&=\sum_{\substack{l_1,\cdots,l_k\ge 0\\ x+l\cdot v- a\in p^{n+t}\cali{O}_{V^\dagger,p}}} u^{x+l\cdot v} \xi(x+l\cdot v)\\
			&=\sum_{y\in R_{V,x}(a,p^n)} \frac{u^y\xi(y)}{\prod_{1\le i\le k}(1-u^{p^nv_i}\xi(p^n v_i))}\\
			&\to \sum_{y\in R_{V,x}(a,p^n)} \frac{\xi(y)}{\prod_{1\le i\le k}(1-\xi(p^n v_i))}\qquad(u\to 1^-).
		\end{align*}
		Here in the third equality we used the fact that for $z\in \cali{O}_{V^\dagger,p}/p^{n+t}\cali{O}_{V^\dagger,p}$, $\sum_{0\le w<p^{n+t}}\zeta^{\chx{z,w}} = p^{n+t}\1_{z=0}$.
	\end{proof}
	
	\begin{rem}
		\label{rem.R-singleton}
		The set $R_{V,x}(a,p^n)$ is nonempty if and only if $a+p^n\cali{O}_{V,p} = x+l\cdot v+p^n\cali{O}_{V,p}$ for some $0\le l<p^n$, \textit{i.e.}, $a+p^n\cali{O}_{V,p}\subseteq x+\cali{O}_{V,p}$. Suppose this is the case. Then, as $V$ is also a $\Q_p$-basis of $F_p$, $R_{V,x}(a,p^n)$ is the singleton $\{x+l\cdot v\}$.
	\end{rem}
	
	\subsection{Second step}
	
	To state the result below, we need some notation. For any $y\in F$ that is $\cali{N}$-integral, set
	\begin{align*}
		R(y,\cali{N}) = R_V(y,\cali{N}) = \left\{z = \sum_{1\le i\le k}z_iv_i\in F:z_i\in \Z, 0\le z_i<N, z-y\in \cali{N}\right\}.
	\end{align*}

	Also, we define the coefficients $\{b_i\}_{0\le i\le k}\subset \Z_p$ by the expansion:
	\begin{align*}
		N^k\left(\frac{1-u}{1-u^N}\right)^k = b_0 +b_1(u-1)+\cdots+ b_k(u-1)^k + O(u-1)^{k+1};
	\end{align*}
	note that $b_0=1$. Finally, for $z=\sum_{1\le i\le k} z_iv_i\in F$, write $\tilde{z} = \sum_{1\le i\le k} z_i$. 
	\begin{lem}
		\label{lem.coefficients-unknown}
		We have
		\begin{align*}
			\mu_{V,x,\cali{N}}(a+p^n\cali{O}_{V,p}) = \frac{(-1)^k}{N^{k-1}}
			\sum_{y\in R_{V,x}(a,p^n)}
			\sum_{z\in R(-y/p^n,\cali{N})}
			\left[b_k + b_{k-1}\binom{\tilde{z}}{1} +\cdots + b_0\binom{\tilde{z}}{k}\right].
		\end{align*}
	\end{lem}
	
	\begin{proof}
		Again it suffices to prove the identity over $\bar{\Q}$, thus over $\C$, by summing \eqref{equation.period-step1} over all nontrivial $\xi\in (\cali{O}/\cali{N})^\wedge$. Let $u\in (0,1)$ be a real parameter. Then
		\begin{align*}
			\mu_{V,x,\cali{N}}(a+p^n\cali{O}_{V,p}) = \lim_{u\to 1^-} \sum_{\xi\ne 1}\sum_{y\in R_{V,x}(a,p^n)} \frac{\xi(y)}{\prod_{1\le i\le k}(1-u\xi(p^n v_i))}.
		\end{align*}
		Before taking the limit, we have
		\begin{align*}
			\sum_{\xi\ne 1}\sum_{y\in R_{V,x}(a,p^n)} \frac{\xi(y)}{\prod_{1\le i\le k}(1-u\xi(p^n v_i))}
			&= \sum_{\xi\ne 1}\sum_{y\in R_{V,x}(a,p^n)}\sum_{l_1,\cdots,l_k\ge 0} u^{l_1+\cdots+l_k}\xi(y+p^n l\cdot v)\\
			&= \sum_{y\in R_{V,x}(a,p^n)} \left[\sum_{\substack{l_1,\cdots,l_k\ge 0\\ y+p^n l\cdot v\in \cali{N}}} Nu^{l_1+\cdots+l_k} - \frac{1}{(1-u)^k}\right]\\
			&= \sum_{y\in R_{V,x}(a,p^n)} \left[-\frac{1}{(1-u)^k} +N\sum_{z\in R(-y/p^n,\cali{N})}\frac{u^{\tilde{z}}}{(1-u^N)^k}\right].
		\end{align*}
		Since the above rational function is regular at $u=1$, after we take the limit $u\to 1^-$, only the constant term survives. Therefore it boils down to computing the degree zero term of
		\begin{align*}
			N\frac{u^{\tilde{z}}}{(1-u^N)^k}&= \frac{1}{N^{k-1}(1-u)^k}\cdot u^{\tilde{z}} N^k\left(\frac{1-u}{1-u^N}\right)^k\\
			&=\frac{1}{N^{k-1}(1-u)^k} \sum_{0\le i\le k} \binom{\tilde{z}}{i}(u-1)^i
			\sum_{0\le j\le k}b_j (u-1)^j + O(u-1),
		\end{align*}
		which is clearly $\frac{(-1)^k}{N^{k-1}}\sum_{0\le i\le k}b_i\binom{\tilde{z}}{k-i}$.
	\end{proof}
	
	\subsection{Final step}
	
	\begin{lem}
		\label{lem.sum-independence}
		For all $0\le i\le k$, there exist polynomials $P_i(X)\in \Q[X]$, such that for any Cassou-Nogu\`es ideal $\cali{N}$ and any $y\in \cali{O}/\cali{N}$,
		\begin{enumerate}
			\item[a)] if $0\le i< k$, then
			\begin{align*}
				\sum_{z\in R(y,\cali{N})} \binom{\tilde{z}}{i} = P_i(N);
			\end{align*}
		
			\item[b)] if $i=k$, then
			\begin{align*}
				\sum_{z\in R(y,\cali{N})} \binom{\tilde{z}}{k} = \sum_{\substack{1\le d_1,d_2,\cdots,d_k<N\\ d_1v_1+\cdots+d_kv_k\equiv y \bmod \cali{N}}}d_1d_2\cdots d_k + P_k(N).
			\end{align*}
		\end{enumerate}
	\end{lem}
	
	\begin{proof}
		Write $z = z_1v_1+\cdots+z_kv_k$, so $\tilde{z} = z_1+\cdots+z_k$. We have
		\begin{align*}
			\binom{z_1+\cdots+z_k}{i} = \sum_{i_1+\cdots+i_k=i} \binom{z_1}{i_1}\binom{z_2}{i_2}\cdots\binom{z_k}{i_k},
		\end{align*}
		where the sum is over all nonnegative tuples $(i_1,\cdots,i_k)$ with $i_1+\cdots+i_k = i$. Thus the study of the sum $\sum_{z\in R(y,\cali{N})} \binom{\tilde{z}}{i}$ boils down to that of $\sum_{z\in R(y,\cali{N})} \binom{z_1}{i_1}\cdots\binom{z_k}{i_k}$ for each tuple $(i_1,\cdots,i_k)$. 
		
		\vspace{3mm}
		
		Assume first some $i_r=0$; without loss of generality say $r=k$. Note that this assumption is automatic if $i<k$. In this case, consider the following parametrization of $R(y,\cali{N})$:
		\begin{align*}
			\left\{d_1v_1+d_2v_2+\cdots+d_{k-1}v_{k-1} +\left(\frac{y-d_1v_1-\cdots-d_{k-1}v_{k-1}}{v_k}\right)^\flat_{\cali{N}}v_k \in F: 0\le d_1,d_2,\cdots,d_{k-1}<N\right\}.
		\end{align*}
		Using this, we find
		\begin{align*}
			\sum_{z\in R(y,\cali{N})}\binom{z_1}{i_1}\cdots\binom{z_k}{i_k}&=\sum_{z\in R(y,\cali{N})}\binom{z_1}{i_1}\cdots\binom{z_{k-1}}{i_{k-1}}\\
			&=\sum_{0\le d_1,\cdots,d_{k-1}<N} \binom{d_1}{i_1}\cdots\binom{d_{k-1}}{i_{k-1}}\\
			&=\binom{N}{i_1+1}\cdots\binom{N}{i_{k-1}+1}.
		\end{align*}
		As such, for $0\le i<k$, we conclude that
		\begin{align*}
			P_i(X) = \frac{1}{X}\sum_{i_1+\cdots+i_k=i} \binom{X}{i_1+1}\cdots\binom{X}{i_k+1}.
		\end{align*}
		
		\vspace{3mm}
		
		As for the remaining case that none of $i_r$ is $0$, we must have $i=k$ and $i_1=i_2\cdots=i_k = 1$. In turn, 
		\begin{align*}
			\sum_{z\in R(y,\cali{N})} \binom{z_1}{1}\cdots\binom{z_k}{1} = \sum_{\substack{1\le d_1,d_2,\cdots,d_k<N\\ d_1v_1+\cdots+d_kv_k\equiv y \bmod \cali{N}}}d_1d_2\cdots d_k.
		\end{align*}
		It also follows that
		\begin{align*}
			P_k(X) = \frac{1}{X}\sum_{\substack{i_1+\cdots+i_k=k\\i_1i_2\cdots i_k= 0}} \binom{X}{i_1+1}\cdots\binom{X}{i_k+1}.
		\end{align*}
	\end{proof}
	
	Combining Lemmas \ref{lem.coefficients-unknown} and \ref{lem.sum-independence}, we conclude that 
	\begin{align*}
		\mu_{V,x,\cali{N}}(a+p^n\cali{O}_{V,p}) = \frac{(-1)^k}{N^{k-1}}\sum_{y\in R_{V,x}(a,p^n)} \left[\sum_{\substack{1\le d_1,d_2,\cdots,d_k<N\\ d_1v_1+\cdots+d_kv_k\equiv -y/p^n \bmod \cali{N}}} d_1d_2\cdots d_k + P(N)\right],
	\end{align*}
	for some polynomial $P(X)$ independent of $a,n,y$. To ease notation, denote by
	\begin{align}
		H_V(y) = \sum_{\substack{1\le d_1,d_2,\cdots,d_k<N\\ d_1v_1+\cdots+d_kv_k\equiv -y \bmod \cali{N}}} d_1d_2\cdots d_k.
	\end{align}
	
	\begin{lem}
		For the polynomial $P(X)$ as above, we have
		\begin{align*}
			P(N) = -N^{k-1}\left(\frac{N-1}{2}\right)^k.
		\end{align*}
	\end{lem}
	
	\begin{proof}
		Recall $q>1$ denotes a power of $p$ with $q\equiv 1\bmod \cali{N}$. Using the additivity of $\mu_{V,x,\cali{N}}$ as a measure, we have
		\begin{align*}
			\mu_{V,x,\cali{N}}(x+\cali{O}_{V,p}) = \sum_{0\le l_1,\cdots,l_k<q} \mu_{V,x,\cali{N}}(x+l\cdot v + q\cali{O}_{V,p}).
		\end{align*}
		Bearing Remark \ref{rem.R-singleton} in mind, we have
		\begin{align*}
			\mu_{V,x,\cali{N}}(x+l\cdot v + q^{\epsilon}\cali{O}_{V,p}) = \frac{(-1)^k}{N^{k-1}}[H_V(x+l\cdot v) +P(N)],
		\end{align*}
		for $\epsilon \in\{0,1\}$ and $0\le l_1,\cdots,l_k<q^\epsilon$. Therefore
		\begin{align*}
			H_V(x) = \sum_{0\le l_1,\cdots,l_k<q} H_V(x+l\cdot v) + (q^k-1)P(N).
		\end{align*}
		A small computation shows
		\begin{align*}
			\sum_{0\le l_1,\cdots,l_k<q}H_V(x+l\cdot v) &= \sum_{0\le l_1,\cdots,l_k<q} \sum_{0\le d_1,\cdots,d_{k-1}<N} d_1\cdots d_{k-1}\left(\frac{-x-l\cdot v-d_1v_1-\cdots -d_{k-1}v_{k-1}}{v_k}\right)^\flat_{\cali{N}}\\
			&=H_V(x) + \sum_{0\le d_1,\cdots,d_{k-1}<N} d_1d_2\cdots d_{k-1}\frac{q-1}{N}\frac{N(N-1)}{2}(q^{k-1} +q^{k-2}+\cdots+1)\\
			&=H_V(x) + N^{k-1}\left(\frac{N-1}{2}\right)^k (q^k-1).
		\end{align*}
		Combine the equations and we find $P(N) = -N^{k-1}(\frac{N-1}{2})^k$.
	\end{proof}

	In summary, we have established the first period formula below.
	
	\begin{thm}
		\label{thm.first-period-formula}
		Let $F/\Q$ be a totally real number field of degree $k$, $p$ be a rational prime and $\cali{N}\ne \cali{O}$ be a Cassou-Nogu\`es ideal prime to $p$. Let $x\in F$ be $\cali{N}p$-integral and $V=\{v_1,\cdots,v_k\}\subset\cali{O}_+$ be a $\Q$-basis of $F$, all of whose elements are prime to $\cali{N}$. Then the attached $p$-adic measure $\mu_{V,x,\cali{N}}$ on $\cali{O}_{V^\dagger,p}$ is valued in $\Z[1/N]$, and we have
		\begin{align}
			\label{equation.period-zeta}
			\mu_{V,x,\cali{N}}(x+l\cdot v+p^n\cali{O}_{V,p}) = (-1)^k\left[\frac{1}{N^{k-1}}H_V\left(\frac{x+l\cdot v}{p^n}\right) - \left(\frac{N-1}{2}\right)^k\right],
		\end{align}
		for all $n\in \Z_{\ge 0}$ and $l= (l_1,\cdots,l_k)$ with $0\le l_1,\cdots,l_k<p^n$.
	\end{thm}
	
	\begin{rem}
		When $F=\Q$ and $x=0$, this specializes to the period formula of the regularized Bernoulli measure (see, e.g., \cite[Theorem 3.1]{KW21}):
		\begin{align*}
			\mu_{1,N^{-1}}(a+p^n\Z_p) = -(-a/p^n)^\flat_N + \frac{N-1}{2}.
		\end{align*}
	\end{rem}
	
	\begin{cor}
		Let the assumptions be as in Theorem \ref{thm.first-period-formula}. Let further $\psi$ be a finite character of $\cl_+(p^\infty)$. Then
		\begin{align}
			\begin{split}
				&(1-\psi(\cali{N})\chx{N}^{1-s})L_{F,p}(s,\psi\omega_F)\\
				=& (-1)^{k-1}\lim_{n\to\infty} \sum_{1\le i\le h} \frac{\psi(\got{a}_i)}{\chx{\nm(\got{a}_i)}^s}\sum_{V} \sum_{x\in P(V)\cap \got{a}_i^{-1}} \sum_{\substack{0\le l_1,\cdots,l_k< q^n\\ \gcd(p,x+l\cdot v) = 1}} \frac{H_V(x+l\cdot v)}{N^{k-1}}
				\frac{\psi(x+l\cdot v)}{\chx{\nm(x+l\cdot v)}^s}.
			\end{split}
		\end{align}
	\end{cor}
	
	\begin{proof}
		This follows from the integral representation \eqref{equation.integral-representation-zeta}, the pushforward formula \eqref{equation.push1}, and the vanishing of $\lim_{n\to\infty}\sum_{\substack{0\le l_1,\cdots,l_k< q^n\\ \gcd(p,x+l\cdot v) = 1}} \psi(x+l\cdot v)\chx{\nm(x+l\cdot v)}^{-s}$ (\textit{cf.~}the proof of Lemma \ref{lem.p-analytic-vanishing}).
	\end{proof}

	\begin{rem}
		When $F\ne \Q$, $s=0$ and $\psi$ is the trivial character, through the interpolation property \eqref{equation.interpolation} and the vanishing of $\zeta_F(0)$, we derive the following curious identity
		\begin{align*}
			\sum_{1\le i\le h}\sum_{V}\sum_{x\in P(V)\cap \got{a}_i^{-1}} \left[\frac{1}{N^{k-1}}H_V(x) - \left(\frac{N-1}{2}\right)^k\right]= 0.
		\end{align*}
		For a simple example take $F=\Q(\sqrt{5})$, for which we have $h=|\cl_+(1)|=1$. As such, take $\{\got{a}_i\}_{1\le i\le h}=\{\cali{O}\}$ and $V=\{1,\varepsilon = \frac{3+\sqrt{5}}{2}\}$. The set $P(V)\cap \cali{O}$ is then the singleton $\{1\}$. We have thus proved the following
	\end{rem}
	
	\begin{prop}
		Suppose $N>1$ is such that $X^2 - 3X + 1$ has a solution $\varepsilon$ modulo $N$. Then
		\begin{align*}
			\frac{1}{N}\sum_{1\le d<N} d(-1-d\varepsilon)^\flat_N = \frac{(N-1)^2}{4}.
		\end{align*}
	\end{prop}
	
	It would be interesting to give this an elementary proof.
	
	\begin{rem}\label{rem:general-sumexpr}
		For the duration of this remark let $\chi$ be a Hecke character of an arbitrary conductor $\got{f}\subseteq \cali{O}$. Below we briefly sketch how to establish a sum expression of $L_{F,p}(s,\chi\omega_F)$ using the period formula \eqref{equation.period-zeta}, at the price of working with an auxiliary Cassou-Nogu\`es ideal $\cali{N}$. We will loosely follow the original strategy of \cite{CN79}. First, we need to upgrade the assumptions (A1) and (A3) in §\ref{subsection.intro.main-results} to
		\begin{enumerate}
			\item[(A1\textsuperscript{+})] The ideal $\cali{N}$ is Cassou-Nogu\`es, not equal to $\cali{O}$ and prime to $p\got{f}$.
			
			\item[(A3\textsuperscript{+})] For all $V$ in the Shintani decomposition and each $v_i\in V$, $v_i$ is prime to $\cali{N}$ and $v_i\in \got{f}_+$.
		\end{enumerate}
		As before, we also assume 
		\begin{enumerate}
			\item[(A2)] All representatives $\got{a}_i$ of $\cl_+(1)$ are integral and prime to $p\cali{N}$.
		\end{enumerate}
		Note that if (A3) is satisfied, we can achieve (A3\textsuperscript{+}) by rescaling all of $V$'s simultaneously by an element in $\got{f}_+\setminus \cali{N}$ if needed, so that the decomposition \eqref{eq:shintani-decomp} still holds.
		
		\vspace{3mm}
		
		Under these assumptions, for $\xi\in\mu_N$ and $\got{a}$ a fractional ideal, consider the complex functions
		\begin{align*}
			L_{V,x,\xi}(s) = \sum_{l\ge 0}\frac{\xi^{x+l\cdot v}}{\nm(x+l\cdot v)^s}
		\end{align*}
		and 
		\begin{align*}
			L(\got{a},s,\chi) = \sum_{\substack{0\ne \got{b}\subseteq \cali{O}\\ [\got{b}]=[\got{a}]\in \cl_+(1)}} \frac{\chi(\got{b})}{\nm(\got{b})^s},
		\end{align*}
		where if $\got{c}$ is a fractional ideal, $[\got{c}]$ denotes its class in $\cl_+(1)$. Then, if $\got{a}=\got{a}_i$ for some $i$, we have (\textit{cf.~}\cite[Th\'eor\`eme 4]{CN79}):
		\begin{align*}
			N^{1-s}\chi(\cali{N})L(\got{a}\cali{N}^{-1},s,\chi)- L(\got{a},s,\chi) = \nm(\got{a})^{-s}\sum_{V}\sum_{x\in P(V)\cap \got{a}^{-1}} \chi(\got{a}x) \sum_{\xi\ne 1,\xi^N=1} L_{V,x,\xi}(s).
		\end{align*}
		Therefore,
		\begin{align*}
			-(1-\chi(\cali{N})N^{1-s})L_F(s,\chi) = \sum_{1\le i\le h} \nm(\got{a}_i)^{-s}
			\sum_V\sum_{x\in P(V)\cap \got{a}_i^{-1}} \chi(\got{a}_ix)\sum_{\xi\ne 1,\xi^N=1} L_{V,x,\xi}(s).
		\end{align*}
		Building on this, and removing the summands of $L_{V,x,\xi}$ such that $\gcd(x+l\cdot v,p)\ne 1$ if necessary, we can run the interpolation argument in \cite[§IV]{CN79} to obtain
		\begin{align*}
			-(1-\chi\omega_F(\cali{N})\chx{N}^{1-s})L_{F,p}(s,\chi\omega_F) = \sum_{1\le i\le h}\chx{\nm(\got{a}_i)}^{-s} \sum_V\sum_{x\in P(V)\cap \got{a}_i^{-1}} \chi(\got{a}_ix)L_{p,V,x,\cali{N}}(s,\omega_F),
		\end{align*}
		where $L_{p,V,x,\cali{N}}(s,\omega_F) =\int_{\cali{O}_p^\times}\chx{\nm\alpha}^{-s}\mu_{V,x,\cali{N}}(\alpha)$. Thus the sum expression of $-(1-\chi\omega_F(\cali{N})\chx{N}^{1-s})L_{F,p}(s,\chi\omega_F)$ would follow from Theorem \ref{thm.first-period-formula}.
	\end{rem}

	\section{Explicit period formula: the Dirichlet case}
	\label{section.period2}
	
	Retain the setting of the last section. Let further $\chi$ be a finite Hecke character on $\cl_+(\cali{N})$ of nontrivial narrow modulus; that is to say, $\chi$ is nontrivial on the image of $(\cali{O}/\cali{N})^\times$ from the canonical exact sequence
	\begin{align*}
		E \to (\cali{O}/\cali{N})^\times \to \cl_+(\cali{N}) \to \cl_+(1)\to 1.
	\end{align*}
	In this case, the rational function is given by
	\begin{align*}
		f_{V,x,\chi}(t) = \sum_{0\le d_1,\cdots,d_k<N} \chi(x+d\cdot v)\frac{t^{x+d\cdot v}}{\prod_{1\le i\le k}(1-t^{Nv_i})},
	\end{align*}
	which \textit{a priori} lives in the fraction field of $\cali{A}_{V^\dagger}$. As a preliminary, we prove
	\begin{lem}
		For all $1\le i\le k$ and any $\zeta$ a $p^t$-th root of unity, the function $f_{V,x,\chi}(t)$ is regular at $t_i=\zeta$. Thus $f_{V,x,\chi}$ belongs to $\cali{A}_{V^\dagger}$.
	\end{lem}
	
	\begin{proof}
		Since $1-t^{Nv_i} = 1-t_i^{Np^t}$, the vanishing order of the denominator of $f_{V,x,\chi}$ at $t_i=\zeta$ is exactly 1. Thus it suffices to show the numerator also has vanishing order $\ge 1$ at $t_i=\zeta$. Without loss of generality suppose $i=k$. We have
		\begin{align*}
			\sum_{0\le d_1,\cdots,d_k<N} \chi(x+d\cdot v) t^{x+d\cdot v} &= \sum_{0\le d_1,\cdots,d_{k-1}<N}t^{x+\sum_{1\le i\le k-1}d_iv_i}\sum_{0\le d_k<N}\chi(x+d\cdot v)t_k^{p^td_k}\\
			&= \sum_{0\le d_1,\cdots,d_{k-1}<N}t^{x+\sum_{1\le i\le k-1}d_iv_i}\sum_{0\le d_k<N}\chi(x+d\cdot v) + (t_k-\zeta)g(t)\\
			&=(t_k-\zeta)g(t),
		\end{align*}
		for some $g(t)\in\cali{A}_{V^\dagger}$.
		
		\vspace{3mm}
		
		We now show that this implies that $f_{V,x,\chi}(t)\in \cali{A}_{V^\dagger}$. Note first that
		\begin{align*}
			f_{V,x,\chi}(t^{1/N}) = \sum_{0\le d<N} \chi(x+d\cdot v) \frac{t^{\frac{x+d\cdot v}{N}}}{\prod_{1\le i\le k}(1-t_i^{p^t})}.
		\end{align*}
		Since $\gcd(N,p)=1$, we see that for all $1\le i\le k$, $t_i^{1/N}\in \cali{A}_{V^\dagger}$, so $f_{V,x,\chi}(t^{1/N})$ belongs to $\mathrm{Frac}(\cali{A}_{V^\dagger})$. Let $[N]$ be the multiplication-by-$N$ endomorphism of $\mathrm{Frac}(\cali{A}_{V^\dagger})$ induced from that of the formal torus of $\cali{A}_{V^\dagger}$. Then, since $p\nmid N$, $[N]$ is an isomorphism and preserves $\cali{A}_{V^\dagger}$. As such, noticing the image of $f_{V,x,\chi}(t^{1/N})$ under $[N]$ is $f_{V,x,\chi}(t)$, we see that the only possible poles of $f_{V,x,\chi}(t)$ are the $N$-multiples of those of $f_{V,x,\chi}(t^{1/N})$. Since the latter is contained in $\mu_{p^t}$, we conclude that $f_{V,x,\chi}$ is regular by the first part.
	\end{proof}
	
	Since $\mu_{V,x,\chi}$ is the measure corresponding to $f_{V,x,\chi}$, we may apply Proposition \ref{prop.period-formula-setup} to compute the period
	\begin{align}
		\label{equation.period-formula-dirichlet-setup}
		\mu_{V,x,\chi}(a+p^n\cali{O}_{V,p}) = \frac{1}{p^{(n+t)k}} \sum_{\varphi\in(\cali{O}_{V^\dagger,p}/p^{n+t})^\wedge} \varphi^{-1}(a)f_{V,x,\chi}|_{\varphi}.
	\end{align}
	Below we shall establish the explicit formula of $\mu_{V,x,\chi}$ in several steps. Towards this, we introduce a stratification of $(\cali{O}_{V^\dagger,p}/p^{n+t})^\wedge$. As in the proof of Lemma \ref{lem.period-step1}, we identify $\Hom(\cali{O}_{V^\dagger,p}/p^{n+t},\Z/p^{n+t})$ with $(\Z/p^{n+t})^k$. For any subset $S\subseteq \{1,2,\cdots,k\}$, we define
	\begin{align*}
		\cali{P}_S =\{w\in (\Z/p^{n+t})^k: p^tw_i=0 \text{ for all }i\in S, p^tw_i\ne 0\text{ otherwise}\}.
	\end{align*}
	Clearly $(\Z/p^{n+t})^k = \bigsqcup_S \cali{P}_S$. As a consequence, we can dismantle formula \eqref{equation.period-formula-dirichlet-setup} into
	\begin{align*}
		\mu_{V,x,\chi}(a+p^n\cali{O}_{V,p}) = \frac{1}{p^{(n+t)k}} \sum_{S\subseteq\{1,\cdots,k\}}\sum_{\varphi\in\cali{P}_S} \varphi^{-1}(a) f_{V,x,\chi}|_\varphi,
	\end{align*}
	and we denote the $S$-piece $\frac{1}{p^{(n+t)k}}\sum_{\varphi\in\cali{P}_S} \varphi^{-1}(a) f_{V,x,\chi}|_\varphi$ by $\Omega_S(a,n)$.
	
	\subsection{The initial case}
	\label{subsection.period2.empty-S}
	
	We start with $S=\varnothing$, so $\cali{P}_{S}=\cali{P}_{\varnothing}= \{w\in (\Z/p^{n+t})^k: p^tw_i\ne 0\text{ for all }1\le i\le k\}$. Recall for $n\in \Z_{>0}$ and $a\in \Z_p$, we denote by $a^\flat_{p^n}$ the unique integer in $[0,p^n)$ such that $a^\flat_{p^n}\equiv a\bmod p^n$; we extend it further to $a\in\Q_p$ by putting $a^\flat_{p^n}=0$ if $a\notin\Z_p$. Our computation in the $S=\varnothing$ case will be based on the following elementary lemma (the notation introduced only exists therein):
	\begin{lem}
		\label{lem.period-computation}
		Let $p$ be a prime and $N\in \Z_{>1}$ be prime to $p$. Let $k\in \Z_{>0}$, $v\in ((\Z/N)^\times)^k$, $z\in (\Z/N)^k$ and $y\in \Z_p^k$. Let further $\chi$ be a nontrivial Dirichlet character on $(\Z/N)^\times$, and $\zeta$ be a $p$-power root of unity with order $p^{n+t}$ for some $n,t\in \Z_{\ge 0}$. Then
		\begin{align*}
			\frac{1}{p^{(n+t)k}}\sum_{w \in \cali{P}_{\varnothing}} \sum_{0\le d_1,\cdots,d_k<N} \chi(z+d\cdot v) \frac{\zeta^{(y+dp^t)\cdot w}}{\prod_{1\le i\le k}(1-\zeta^{w_i Np^t})} = \frac{(-1)^k}{p^{nk}}\sum_{0\le d_1,\cdots,d_k<N}\chi(z+d\cdot v)\prod_{1\le i\le k}\left(-\frac{y_i+d_ip^t}{Np^t}\right)^\flat_{p^n}.
		\end{align*}
	\end{lem}
	
	\begin{proof}
		It suffices to prove the equality in $\C$. For this let $u\in (0,1)$ be an auxiliary real parameter. Then
		\begin{align*}
			&\frac{1}{p^{(n+t)k}}\sum_{w \in \cali{P}_{\varnothing}} \sum_{0\le d_1,\cdots,d_k<N} \chi(z+d\cdot v) \frac{\zeta^{(y+dp^t)\cdot w}}{\prod_{1\le i\le k}(1-u\zeta^{w_i Np^t})}\\
			=&\frac{1}{p^{(n+t)k}}\sum_{0\le d<N}\chi(z+d\cdot v) \sum_{w \in \cali{P}_{\varnothing}} \sum_{l_1,\cdots,l_k\ge 0} u^{l_1+\cdots+l_k} \zeta^{(y+dp^t+Np^tl)\cdot w}\\
			=&\frac{1}{p^{(n+t)k}}\sum_{0\le d<N} \chi(z+d\cdot v) \sum_{w \in (\Z/p^{n+t})^k} \sum_{l\ge 0} u^{l_1+\cdots+l_k} \zeta^{(y+dp^t+Np^tl)\cdot w}.
		\end{align*}
		Here the last equality follows from the orthogonality of $\chi$: If $w\in (\Z/p^{n+t})^k\setminus \cali{P}_{\varnothing}$ then without loss of generality we may assume $p^tw_k=0$. As such, for all $0\le d_1,\cdots,d_{k-1}<N$,
		\begin{align*}
			\sum_{0\le d_k<N} \chi(z+d\cdot v) \zeta^{(y+dp^t+Np^tl)\cdot w} = \zeta^{(y+Np^tl)\cdot w}\zeta^{p^t\sum_{1\le i\le k-1}d_iw_i} \sum_{0\le d_k<N} \chi(z+d\cdot v) = 0.
		\end{align*}
		Continuing with our computation, we have
		\begin{align*}
			&\frac{1}{p^{(n+t)k}}\sum_{0\le d<N} \chi(z+d\cdot v)\sum_{l\ge 0}u^{l_1+\cdots+l_k}\sum_{w \in (\Z/p^{n+t})^k} \zeta^{(y+dp^t+Np^tl)\cdot w}\\
			=&\sum_{0\le d<N} \chi(z+d\cdot v)
			\sum_{\substack{l\ge 0\\ y+dp^t+Np^tl\equiv 0\bmod p^{n+t}}} u^{l_1+\cdots+l_k}\\
			=&\sum_{0\le d<N}\chi(z+d\cdot v) \prod_{1\le i\le k}\frac{u^{(-\frac{y_i+d_ip^t}{Np^t})^\flat_{p^n}}}{1-u^{p^n}}\\
			=&\sum_{0\le d<N}\chi(z+d\cdot v) \prod_{1\le i\le k}\frac{(-\frac{y_i+d_ip^t}{Np^t})^\flat_{p^n}(u-1)+O(u-1)^2}{(1-u)p^n+O(u-1)^2},
		\end{align*}
		where in the last equality we used the regularity of the rational function at $u=1$. The result then follows by letting $u\to 1^-$.
	\end{proof}
	
	We may now compute the $S=\varnothing$-piece using the lemma. Below denote by $\zeta$ a primitive $p^{n+t}$-th root of unity, and for $a\in \cali{O}_{V^\dagger,p}$, write $a=\sum_{1\le i\le k}a_iv_i$ for some $a_i\in p^{-t}\Z_p$. Then
	\begin{align*}
		\Omega_\varnothing(a,n) &=\frac{1}{p^{(n+t)k}}\sum_{w \in \cali{P}_{\varnothing}}\zeta^{-\pair{a}{w}} \sum_{0\le d_1,\cdots,d_k<N}\chi(x+d\cdot v)\frac{\zeta^{\pair{x+d\cdot v}{w}}}{\prod_{1\le i\le k}(1-\zeta^{Np^t w_i})}\\
		&=\frac{(-1)^k}{p^{nk}}\sum_{0\le d_1,\cdots,d_k<N}\chi(x+d\cdot v)\prod_{1\le i\le k}\left(\frac{a_i-x_i-d_i}{N}\right)^\flat_{p^n}.
	\end{align*}

	\subsection{General cases via reduction}
	\label{subsection.general-S}
	
	Now we consider $S\subseteq \{1,2,\cdots,k\}$ that is nonempty. Write $V^{\dagger,S} = \{v'_j\}_{j\notin S}$, and $\cali{O}^S_{V^{\dagger},p} = \oplus_{j\notin S} \Z_pv'_j$. Also, for any $\varphi\in (\cali{O}_{V^\dagger,p})^\wedge\simeq \ilim_n (\Z/p^n)^k$, write $\varphi^S = \varphi|_{\cali{O}^S_{V^{\dagger},p}}$, which is also defined on $\cali{O}_{V^\dagger,p}$ via the canonical projection $\cali{O}_{V^\dagger,p}\to \cali{O}^S_{V^\dagger,p}$. For any $T\subseteq \{1,2,\cdots,k\}\setminus S$, denote by $\cali{P}_{T}^S$ the substratum of $\prod_{j\notin S}(\Z/p^{n+t})$:
	\begin{align*}
		\textstyle
		\cali{P}_{T}^S = \{(w_j)_{j\not\in S}\in \prod_{j\notin S}(\Z/p^{n+t}):p^tw_i=0\text{ for all }i\in T,p^tw_i\ne 0\text{ otherwise}\}.
	\end{align*}
	Thus if $\varphi\in \cali{P}_S$ then $\varphi^S\in \cali{P}^S_{\varnothing}$. For any $g\in\cali{A}_{V^{\dagger,S}}=R[[t_j-1]]_{j\notin S}\subset \cali{A}_{V^\dagger}$ and for any $\varphi\in(\cali{O}_{V^\dagger,p})^\wedge$, we have compatible evaluations $g|_\varphi = g|_{\varphi^S}$. Finally for notational convenience we set $\cali{R} = \{x+d\cdot v: 0\le d_1,\cdots,d_k<N\}$ and $\cali{R}^S = \{y^S = \sum_{j\notin S}y_jv_j : y = \sum_{0\le i\le k} y_iv_i \in \cali{R}\}$. So $\cali{R} = \{y^S + \sum_{i\in S} (x_i+d_i)v_i: y^S\in \cali{R}^S, 0\le d_i<N \text{ for all } i\in S\}$.
	
	\vspace{3mm}
	
	Now let $\varphi\in \cali{P}_S$. Note that
	\begin{align*}
		f_{V,x,\chi}|_\varphi&=
		\left.\left(\sum_{y\in\cali{R}} \chi(y) \frac{t^y}{\prod_{1\le i\le k}(1-t_i^{Np^t})}\right)\right|_\varphi\\
		&= \sum_{y^S\in \cali{R}^S} \left.\frac{t^{y^S}}{\prod_{j\notin S}(1-t_j^{Np^t})}\right|_{\varphi^S} \left.\left(\sum_{\substack{0\le d_i<N\\ i\in S}}\chi(y) \prod_{i\in S}\frac{t_i^{p^t(x_i+d_i)}}{1-t_i^{Np^t}}\right)\right|_{t_i=\zeta_i,i\in S},
	\end{align*}
	where $y = y^S + \sum_{i\in S} (x_i+d_i)v_i$ and $\zeta_i = \varphi(v'_i)$ is some $p^t$-th root of unity for $i\in S$. Using the regularity at $t_i=\zeta_i$ for all $i\in S$, we have
	\begin{align*}
		\frac{\sum_{\substack{0\le d_i<N\\ i\in S}}\chi(y) \prod_{i\in S}t_i^{p^t(x_i+d_i)}}{\prod_{i\in S}(1-t_i^{Np^t})} 
		= \frac{\left[\sum_{\substack{0\le d_i<N\\ i\in S}} \chi(y) \prod_{i\in S}p^t(x_i+d_i)\zeta_i^{p^tx_i-1}\right]\prod_{i\in S}(t_i-\zeta_i) + G_1(t)\prod_{i\in S}(t_i-\zeta_i)}{\prod_{i\in S}Np^t\zeta_i^{-1}\prod_{i\in S}(\zeta_i-t_i) + G_2(t)\prod_{i\in S}(t_i-\zeta_i)},
	\end{align*}
	where $G_i(t)|_{t_i=\zeta_i,i\in S} =0$ for $i=1,2$. This shows
	\begin{align*}
		f_{V,x,\chi}|_\varphi
		= \frac{(-1)^{|S|}}{N^{|S|}}
		\sum_{\substack{0\le d_i<N\\ i\in S}}\prod_{i\in S} (x_i+d_i)\zeta_i^{p^tx_i} \sum_{y^S\in \cali{R}^S} \chi(y)\left.\frac{t^{y^S}}{\prod_{j\notin S}(1-t_j^{Np^t})}\right|_{\varphi^S}.
	\end{align*}
	In turn, taking a primitive $p^{n+t}$-th root of unity $\zeta$, we have
	\begin{align*}
		&\Omega_S(a,n)\\
		=&\frac{1}{p^{(n+t)k}}\sum_{\varphi\in \cali{P}_S} \varphi^{-1}(a)f_{V,x,\chi}|_\varphi \\
		=&\frac{1}{p^{(n+t)k}}\sum_{\substack{w_i\in p^n\Z/p^{n+t}\\ i\in S}} \sum_{\varphi^S\in \cali{P}^S_{\varnothing}}
		\zeta^{-p^t\sum_{i\in S}a_iw_i}(\varphi^S)^{-1}(a)
		\frac{(-1)^{|S|}}{N^{|S|}}
		\sum_{\substack{0\le d_i<N\\ i\in S}}\prod_{i\in S} (x_i+d_i)\zeta^{p^tx_iw_i} \sum_{y^S\in \cali{R}^S} \chi(y)\left.\frac{t^{y^S}}{\prod_{j\notin S}(1-t_j^{Np^t})}\right|_{\varphi^S}
		\\
		=&\frac{(-1)^{|S|}}{N^{|S|}p^{|S|n}}
		\1_{x_i-a_i\in\Z_p,i\in S}
		\sum_{\substack{0\le d_i<N\\ i\in S}} \prod_{i\in S}(x_i+d_i)
		\frac{1}{p^{(k-|S|)(n+t)}}
		\sum_{\varphi^S\in \cali{P}^S_{\varnothing}}
		(\varphi^S)^{-1}(a)
		\sum_{y^S\in \cali{R}^S}
		\chi(y)\left.\frac{t^{y^S}}{\prod_{j\notin S}(1-t_j^{Np^t})}\right|_{\varphi^S}\\
		=&\frac{(-1)^{|S|}}{N^{|S|}p^{|S|n}}
		\1_{x_i-a_i\in\Z_p,i\in S}
		\sum_{\substack{0\le d_i<N\\ i\in S}}\prod_{i\in S}(x_i+d_i) \frac{(-1)^{k-|S|}}{p^{(k-|S|)n}}
		\sum_{\substack{0\le d_j<N\\ j\notin S}} \chi(x+d\cdot v)\prod_{j\notin S}\left(\frac{a_j-x_j-d_j}{N}\right)^\flat_{p^n}\\
		=&\frac{(-1)^k}{p^{nk}}\1_{x_i-a_i\in\Z_p,i\in S}\sum_{0\le d_1,\cdots,d_k<N} \chi(x+d\cdot v)\prod_{i\in S}\frac{x_i+d_i}{N}\prod_{j\notin S}\left(\frac{a_j-x_j-d_j}{N}\right)^\flat_{p^n}.
	\end{align*}
	
	\begin{rem}
		Recall that for $y\in\Q_p\setminus \Z_p$, we declared $y^\flat_{p^n}$ to be $0$. Hence from the explicit formulas of $\Omega_S(a,n)$ for various $S$, it is clear that for $\mu_{V,x,\chi}(a+p^n\cali{O}_{V,p})\ne 0$, $a\in x+\cali{O}_{V,p}$.
	\end{rem}
	
	\subsection{Ultimate formula}
	
	We can now prove
	\begin{thm}
		\label{thm.second-period-formula}
		Let the notation and assumptions be as in Theorem \ref{thm.first-period-formula}. Let further $\chi: \cl_+(\cali{N})\to \bar{\Q}^\times$ be a character of nontrivial narrow modulus. Then the attached $p$-adic measure $\mu_{V,x,\chi}$ on $\cali{O}_{V^\dagger,p}$ is valued in $\Z[1/N,\im(\chi)]$, and we have
		\begin{align}
			\label{equation.second-period-formula}
			\mu_{V,x,\chi}(x+l\cdot v+p^n\cali{O}_{V,p}) = \frac{(-1)^k}{N^k}\sum_{0\le d_1,\cdots,d_k<N} \chi(x+(l+p^nd)\cdot v)\prod_{1\le i\le k} d_i
		\end{align}
		for all $n\in \Z_{\ge 0}$ and $l=(l_1,\cdots,l_k)$ with $0\le l_1,\cdots,l_k<p^n$. Moreover, when $p^n\equiv 1\bmod \cali{N}$, we have
		\begin{align}
			\label{equation.second-period-formula-q}
			\mu_{V,x,\chi}(x+l\cdot v+p^n\cali{O}_{V,p}) =(-1)^k\sum_{S\subseteq\{1,\cdots,k\}} N^{-|S|}
			\sum_{\substack{0\le d_i<N\\ i\in S}}\sum_{\substack{0\le d_j<l_j\\ j\notin S}} \chi(x+d\cdot v)\prod_{i\in S}d_i.
		\end{align}
	\end{thm}
	
	\begin{proof}
		We have seen that $\mu_{V,x,\chi}(a+p^n\cali{O}_{V,p}) = \sum_{S\subseteq\{1,\cdots,k\}} \Omega_S(a,n)$. In the following let $h(l_j-d_j)$ be the unique integer such that $(\frac{l_j-d_j}{N})^\flat_{p^n} = \frac{l_j-d_j+p^nh(l_j-d_j)}{N}$. By the computations done in §\ref{subsection.period2.empty-S} and §\ref{subsection.general-S}, we find
		\begin{align*}
			\mu_{V,x,\chi}(x+l\cdot v+p^n\cali{O}_{V,p}) &= \frac{(-1)^k}{p^{nk}} \sum_{S\subseteq\{1,\cdots,k\}}
			\sum_{0\le d_1,\cdots,d_k<N} \chi(x+d\cdot v)
			\prod_{i\in S} \frac{x_i+d_i}{N}\prod_{j\notin S}\frac{l_j-d_j+p^nh(l_j-d_j)}{N}\\
			&= \frac{(-1)^k}{p^{nk}N^k} \sum_{S\subseteq\{1,\cdots,k\}}
			\sum_{0\le d_1,\cdots,d_k<N} \chi(x+d\cdot v)
			\prod_{i\in S} d_i\prod_{j\notin S}[-d_j+p^nh(l_j-d_j)]\\
			&= \frac{(-1)^k}{p^{nk}N^k}
			\sum_{0\le d_1,\cdots,d_k<N} \chi(x+d\cdot v)
			\sum_{S\subseteq\{1,\cdots,k\}} \prod_{i\in S} d_i\prod_{j\notin S}[-d_j+p^nh(l_j-d_j)].
		\end{align*}
		where in the second equality we used the nontriviality of $\chi$ on $(\cali{O}/\cali{N})^\times$. The last term can be vastly simplified by the identity
		\begin{align*}
			\sum_{S\subseteq\{1,\cdots,k\}} \prod_{i\in S} d_i\prod_{j\notin S}[-d_j+p^nh(l_j-d_j)]&=\sum_{S\subseteq\{1,\cdots,k\}}
			\sum_{S'\supseteq S}(-1)^{|S'|-|S|}p^{n(k-|S'|)}
			\prod_{i\in S'}d_i\prod_{j\notin S'} h(l_j-d_j)\\
			&=\sum_{S'\subseteq\{1,\cdots,k\}}(-1)^{|S'|}p^{n(k-|S'|)} \prod_{i\in S'}d_i\prod_{j\notin S'}h(l_j-d_j) \sum_{S\subseteq S'}(-1)^{|S|}.
		\end{align*}
		Since $\sum_{S\subseteq S'}(-1)^{|S|}$ vanishes unless $S'=\varnothing$, we have
		\begin{align*}
			\sum_{S\subseteq\{1,\cdots,k\}} \prod_{i\in S}d_i\prod_{j\notin S}[-d_j+p^nh(l_j-d_j)] = p^{nk}\prod_{1\le i\le k}h(l_i-d_i);
		\end{align*}
		thereby
		\begin{align}
			\label{equation.second-period-formula-intermediate}
			\mu_{V,x,\chi}(x+l\cdot v+p^n\cali{O}_{V,p}) = \frac{(-1)^k}{N^k}\sum_{0\le d_1,\cdots,d_k<N} \chi(x+d\cdot v)\prod_{1\le i\le k}h(l_i-d_i).
		\end{align}
		
		Now, as $0\le l_i<p^n$, $0\le d_i<N$, it can be shown that (see, e.g., \cite[Proposition 3.2.]{Zh22})
		\begin{align*}
			h(l_i-d_i) = \left(\frac{d_i-l_i}{p^n}\right)^\flat_N.
		\end{align*}
		Consequently,
		\begin{align*}
			\mu_{V,x,\chi}(x+l\cdot v+p^n\cali{O}_{V,p}) &= \frac{(-1)^k}{N^k}\sum_{0\le d_1,\cdots,d_k<N} \chi(x+(l+d)\cdot v)\prod_{1\le i\le k} \left(\frac{d_i}{p^n}\right)^\flat_N\\
			&= \frac{(-1)^k}{N^k}\sum_{0\le d_1,\cdots,d_k<N} \chi(x+(l+p^nd)\cdot v)\prod_{1\le i\le k} d_i.
		\end{align*}
		To prove the second formula \eqref{equation.second-period-formula-q}, it suffices to prove it when $0\le l_i<N$ for all $1\le i\le k$, since both sides are periodic in each of $l_i$ with period $N$, as long as $0\le l_i<p^n$. Under this restriction, $h(l_i-d_i) = (\frac{d_i-l_i}{p^n})^\flat_{N} = (d_i-l_i)^\flat_N = d_i-l_i + N\1_{l_i>d_i}$. By \eqref{equation.second-period-formula-intermediate}, we have
		\begin{align*}
			\mu_{V,x,\chi}(x+l\cdot v+p^n\cali{O}_{V,p})&=\frac{(-1)^k}{N^k}
			\sum_{0\le d_1,\cdots,d_k<N}\chi(x+d\cdot v)
			\prod_{1\le i\le k}(d_i-l_i+N\1_{l_i>d_i})\\
			&=\frac{(-1)^k}{N^k}
			\sum_{0\le d_1,\cdots,d_k<N}\chi(x+d\cdot v)
			\prod_{1\le i\le k}(d_i+N\1_{l_i>d_i})\\
			&=(-1)^k\sum_{S\subseteq\{1,\cdots,k\}} N^{-|S|}
			\sum_{\substack{0\le d_i<N\\ i\in S}}\sum_{\substack{0\le d_j<l_j\\ j\notin S}} \chi(x+d\cdot v)\prod_{i\in S}d_i.
		\end{align*}
	\end{proof}
	
	\subsection{The sum expression}
	
	We now prove the second part of Theorem \ref{thm.main}.
	\begin{cor}
		\label{cor.sumexpr-dirichlet}
		Let the notation and assumptions be as in Theorem \ref{thm.second-period-formula}. Let further $\psi$ be a finite character of $\cl_+(p^\infty)$. Then
		\begin{align}
			\label{equation.sumexpr-dirichlet}
			L_{F,p}(s,\chi\psi\omega_F)
			= (-1)^k\lim_{n\to\infty} \sum_{1\le i\le h} \frac{\chi\psi(\got{a}_i)}{\chx{\nm(\got{a}_i)}^s}\sum_{V} \sum_{x\in P(V)\cap \got{a}_i^{-1}} \sum_{\substack{0\le l_1,\cdots,l_k< q^n\\ \gcd(p,x+l\cdot v) = 1}} 
			\sum_{0\le d<l^\flat_N} \chi(x+d\cdot v)
			\frac{\psi(x+l\cdot v)}{\chx{\nm(x+l\cdot v)}^s}.
		\end{align}
	\end{cor}
	
	\begin{proof}
		From the integral representation \eqref{equation.integral-representation-dirichlet}, the pushforward formula \eqref{equation.push2} and the period formula \eqref{equation.second-period-formula-q}, we see that $L_{F,p}(s,\chi\psi\omega_F)$ is approximated by
		\begin{align}
			\label{equation.pre-sum-expression-dirichlet}
			\sum_{1\le i\le h} \frac{\chi\psi(\got{a}_i)}{\chx{\nm(\got{a}_i)}^s}\sum_{V} \sum_{x\in P(V)\cap \got{a}_i^{-1}} \sum_{\substack{0\le l_1,\cdots,l_k< q^n\\ \gcd(p,x+l\cdot v) = 1}}
			\frac{\psi(x+l\cdot v)}{\chx{\nm(x+l\cdot v)}^s}
			(-1)^k\sum_{S\subseteq\{1,\cdots,k\}} N^{-|S|}
			\sum_{\substack{0\le d_i<N\\ i\in S}}\sum_{\substack{0\le d_j<l_j\\ j\notin S}} \chi(x+d\cdot v)\prod_{i\in S}d_i.
		\end{align}
		Suppose $S\ne \varnothing$; without loss of generality say $k\in S$. Then
		\begin{align*}
			&\sum_{\substack{0\le l_1,\cdots,l_k< q^n\\ \gcd(p,x+l\cdot v) = 1}} \frac{\psi(x+l\cdot v)}{\chx{\nm(x+l\cdot v)}^s} \sum_{\substack{0\le d_i<N\\ i\in S}}\sum_{\substack{0\le d_j<l_j\\ j\notin S}} \chi(x+d\cdot v)\prod_{i\in S} d_i\\
			=&\sum_{0\le d_1,\cdots,d_k<N} \chi(x+d\cdot v)\prod_{i\in S} d_i
			\sum_{\substack{0\le l_1,\cdots,l_{k-1}<q^n\\ l_j>d_j,j\notin S}}
			\sum_{\substack{0\le l_k<q^n\\ \gcd(p,x+l\cdot v)=1}} \frac{\psi(x+l\cdot v)}{\chx{\nm(x+l\cdot v)}^s}\\
			\equiv&\ 0\pmod{q^{n-c}},
		\end{align*}
		where in the last equality, $c$ is some constant that only depends on $\psi$ and whose existence is ensured by the lemma below. We thus conclude that, when taking the limit of \eqref{equation.pre-sum-expression-dirichlet}, all the $S$-components die except for $S=\varnothing$, whereby the desired formula follows.
	\end{proof}
	
	\begin{lem}
		\label{lem.p-analytic-vanishing}
		Let $f(x) = \psi(x)\chx{\nm(x)}^{-s}$ with $s\in \Z_p$, and suppose $\psi$ factors through $\cl_+(p^r)$ for some $r\in \Z_{>0}$. Then for all $a,b\in \cali{O}_p$ with $a\equiv b\bmod p^n$ for some $n\ge r$, we have $f(a)\equiv f(b)\bmod p^n$. Consequently for all $n\ge r$ and all $a\in \cali{O}_p$, we have
		\begin{align*}
			\sum_{\substack{0\le m<p^n\\ \gcd(p,a+mv_k)=1}} \psi(a+mv_k)\chx{\nm(a+mv_k)}^{-s} \equiv 0\pmod {p^{n-r}}.
		\end{align*}
	\end{lem}
	
	\begin{proof}
		To prove the first congruence, first take some $e\in \Z_{>0}$ very large, so that both $a+p^e$ and $b+p^e$ are totally positive. Then $\psi(a) = \psi((a+p^e)) = \psi((b+p^e)) = \psi(b)$, since $(a+p^e)/(b+p^e)$ is totally positive and is $\equiv 1\bmod p^r$. As $\nm(a)\equiv \nm(b)\bmod p^n$, it follows that $f(a)\equiv f(b)\bmod p^n$. To prove the second congruence, first note that for $T_a=\{x\in \cali{O}/p: x\equiv -a\bmod \got{p}\text{ for some prime }\got{p}\mid p\}$, we have
		\begin{align*}
			\sum_{\substack{0\le m<p^n\\ \gcd(p,a+mv_k)=1}} f(a+mv_k) = \sum_{\substack{0\le m<p^n\\ m\notin T_{a/v_k}\bmod p}} f(a+mv_k).
		\end{align*}
		Hence if $n>r$, then
		\begin{align*}
			\sum_{\substack{0\le m<p^n\\ m\notin T_{a/v_k}\bmod p}} f(a+mv_k)&= \sum_{0\le i<p}\sum_{\substack{0\le m<p^{n-1}\\ m\notin T_{a/v_k}\bmod p}}
			f(a+(m+ip^{n-1})v_k)\\
			&\equiv \sum_{0\le i<p}\sum_{\substack{0\le m<p^{n-1}\\ m\notin T_{a/v_k}\bmod p}}
			f(a+mv_k) \pmod{p^{n-1}}\\
			&=p\sum_{\substack{0\le m<p^{n-1}\\ m\notin T_{a/v_k}\bmod p}}
			f(a+mv_k) \pmod{p^{n-1}}.
		\end{align*}
		Iterate this process until $n=r$, and we obtain the congruence
		\begin{align*}
			\sum_{\substack{0\le m<p^n\\ m\notin T_{a/v_k}\bmod p}} f(a+mv_k) \equiv p^{n-r}\sum_{\substack{0\le m<p^r\\ m\notin T_{a/v_k}\bmod p}} f(a+mv_k)\pmod {p^{n-1}}.
		\end{align*}
		Since $f$ is valued in $\bar{\Z}_p$, the second congruence follows.
	\end{proof}

	\begin{rem}
		\label{rem.value-at-zero}
		As a sanity check, in the simplest case when $p$ is inert, $\psi$ is trivial and $s=0$, under assumption (A5) we can evaluate the right hand side of \eqref{equation.sumexpr-dirichlet} to $(1-\chi(p))L_F(0,\chi)$, as predicted by the interpolation property \eqref{equation.interpolation}. See Appendix \ref{section.appendix}.
	\end{rem}

	\section{Ferrero-Greenberg type formulas}
	\label{section.ferrero-greenberg}
	
	In this section, using the sum expression \eqref{equation.sumexpr-dirichlet}, we prove a generalization of the classical formula of Ferrero-Greenberg \cite{FG78}. In the rest of this article, we assume in addition
	\begin{align}
		\label{assumption.Op}\tag{A5}
		\cali{O}_p= \cali{O}_{V,p}\text{ for all }V\text{ appearing in the cone decomposition.}
	\end{align}
	To state the formula, we first define the $p$-adic Hecke-Shintani $L$-function
	\begin{align*}
		L_{p,V,x}(s,\chi\psi) = \int_{\cali{O}_p^\times} \psi\omega_F^{-1}(\alpha)\chx{\nm\alpha}^{-s}\mu_{V,x,\chi}(\alpha);
	\end{align*}
	the proof of Corollary \ref{cor.sumexpr-dirichlet} provides the following sum expression
	\begin{align*}
		L_{p,V,x}(s,\chi\omega_F) = (-1)^k\lim_{n\to\infty} 
		\sum_{\substack{0\le l_1,\cdots,l_k< q^n\\ \gcd(p,x+l\cdot v) = 1}} \chx{\nm(x+l\cdot v)}^{-s}\sum_{0\le d<l} \chi(x+d\cdot v).
	\end{align*}
	Next define the multiple $p$-adic Gamma function on $\cali{O}_{V,p}\approx\Z_p^k$ to be
	\begin{align}
		\label{definition.Gamma}
		\Gamma_{F,p,V}(y_1v_1+\cdots+y_kv_k) = \Gamma_{F,p,V}(y_1,\cdots,y_k) = \lim_{n>0,n\to y} \prod_{\substack{1\le l_i<n_i,1\le i\le k\\ \gcd(p,l\cdot v)=1}} \chx{\nm(l\cdot v)}.
	\end{align}
	
	\begin{rem}
		When $F=\Q$ and $V=\{1\}$ we recover Morita's $p$-adic Gamma function up to a root of unity. At the price of indulging this difference, the convergence of the product defining $\Gamma_{F,p,V}$ is straightforward, for the existence of $\lim_{n>0,n\to y}\sum_{\substack{1\le l_i<n_i,1\le i\le k\\ \gcd(p,l\cdot v)=1}}\log_p\nm(l\cdot v)$ is.
	\end{rem}
	
	\begin{prop}
		\label{prop.ferrero-greenberg}
		Let the notation and assumptions be as in Theorem \ref{thm.second-period-formula}; additionally assume \eqref{assumption.Op}. Let $L_{p,V,x}(s,\chi\omega_F)$ be the $p$-adic Hecke-Shintani $L$-function. Then
		\begin{align*}
			L_{p,V,x}'(0,\chi\omega_F) = (-1)^{k-1}\sum_{0\le d_1,\cdots,d_k<N} \chi(x+d\cdot v)\log_p\Gamma_{p,F}\left(\frac{x+d\cdot v}{N}\right) - k\log_p(N)L_{p,V,x}(0,\chi\omega_F).
		\end{align*}
	\end{prop}
	
	The proof is based on some elementary results gathered below. In what follows, assume $q\equiv 1\bmod Nh$.
	
	\begin{lem}
		\label{lem.p-analyse}
		Suppose $x$ is $\cali{N}p$-integral and is of the form $x=\sum_{1\le i\le k}(c_i/h)v_i$ with $h,c_1,\cdots,c_k\in \Z$ and $\gcd(h,c_1,\cdots,c_k)=1$. Then
		\begin{align*}
			&\lim_{n\to\infty}\sum_{\substack{0\le l_1,\cdots,l_k< q^n\\ \gcd(p,x+l\cdot v) = 1}} \chx{\nm(x+l\cdot v)}^{-s}\sum_{0\le d<l} \chi(x+d\cdot v)\\
			=&\lim_{n\to\infty} \sum_{1\le d_1,\cdots,d_k\le N} \chi(x+(d-1)\cdot v)
			\sum_{\substack{\frac{c+q^n(h-c)}{h}\le l<q^n+\frac{c+q^n(h-c)}{h}\\ \gcd(p,l\cdot v)=1, l^\sharp_N>d}} \chx{\nm(l\cdot v)}^{-s}.
		\end{align*}
	\end{lem}
	
	\begin{proof}
		First note that since $x$ is $p$-integral and $V$ is a basis of $\cali{O}_p$, $h$ is prime to $p$. Using the congruences
		\begin{align*}
			x \equiv \sum_{1\le i\le k} \frac{c_i+q^n(h-c_i)}{h}v_i \pmod {q^n}
		\end{align*}
		and
		\begin{align*}
			\frac{c+q^n(h-c)}{h}\equiv \frac{c+(h-c)}{h} = 1\pmod N,
		\end{align*}
		we have
		\begin{align*}
			&\sum_{\substack{0\le l_1,\cdots,l_k< q^n\\ \gcd(p,x+l\cdot v) = 1}} \chx{\nm(x+l\cdot v)}^{-s}\sum_{0\le d<l} \chi(x+d\cdot v)\\
			\equiv&\sum_{\substack{\frac{c+q^n(h-c)}{h}\le l<q^n+\frac{c+q^n(h-c)}{h}\\ \gcd(p,l\cdot v)=1}} \chx{\nm(l\cdot v)}^{-s}
			\sum_{0\le d<l-\frac{c+q^n(h-c)}{h}} \chi(x+d\cdot v) \pmod{q^n}\\
			=&\sum_{1\le d\le N}\chi(x+(d-1)\cdot v) \sum_{\substack{\frac{c+q^n(h-c)}{h}\le l<q^n+\frac{c+q^n(h-c)}{h}\\ \gcd(p,l\cdot v)=1, l^\sharp_N>d}} \chx{\nm(l\cdot v)}^{-s}.	
		\end{align*}
	\end{proof}
	
	\begin{lem}
		\label{lem.ferrero-greenberg-map}
		Given $n\in \Z_{>0}$, $t\in \Z\cap [1,N]$ and $s\in \Z$, write
		\begin{align*}
			\Phi^n_{s/h}(t) = \{m\in \Z: \frac{h-s+q^n s}{h}\le m<q^n+\frac{h-s+q^n s}{h}, m^\sharp_N>t\}
		\end{align*}
		and
		\begin{align*}
			\Psi^n_{s/h}(t) = \{m\in \Z: 1+\frac{s(q^n-1)}{Nh}\le m<1+\frac{s(q^n-1)}{Nh}+(N-t)\frac{q^n-1}{N}\}.
		\end{align*}
		The Ferrero-Greenberg map $\iota: m = m^\sharp + kN\mapsto (k+1)+(N-m^\sharp)\frac{q^n-1}{N}$ gives a bijection $\Phi^n_{s/h}(t)\xrightarrow{\sim}\Psi^n_{s/h}(t)$. Moreover, for $m\in \Phi^n_{s/h}(t)$ we have $m\equiv N\iota(m)\bmod q^n$, so in particular $p\mid m$ if and only if $p\mid \iota(m)$.
	\end{lem}
	
	\begin{proof}
		Let $m = m^\sharp + kN\in \Psi^n_{s/h}(t)$. It is straightforward to show that $k\ge \frac{s(q^n-1)}{Nh}$ and $k\le (1+s/h)\frac{q^n-1}{N}$; the upper bound is unattainable because $1+(1+s/h)\frac{q^n-1}{N}\cdot N = q^n+\frac{-s+q^ns}{h}\equiv 1 \bmod N$, thus not belonging to $\Phi^n_{s/h}(t)$ with $t\ge 1$. Conversely, any tuple $(r,k)$ with $t< r\le N$ and $\frac{s(q^n-1)}{Nh}\le k<\frac{(h+s)(q^n-1)}{Nh}$ gives an element of $\Phi^n_{s/h}(t)$. As such, that $\iota$ is a bijection between $\Phi^n_{s/h}(t)$ and $\Psi^n_{s/h}(t)$ follows from Euclidean division. Finally the congruence $m\equiv N\iota(m) \bmod q^n$ follows from a simple computation in \cite[proof of Lemma 1]{FG78}.
	\end{proof}
	
	\begin{cor}
		\label{cor.big-ferrero-greenberg-map}
		Let $c$ and $h$ be as in Lemma \ref{lem.p-analyse}. Given any $d=(d_1,\cdots,d_k)$ with $1\le d_1,\cdots,d_k\le N$, the Ferrero-Greenberg map induces a bijection from 
		\begin{align*}
			\left\{l\in \Z^k: \frac{c+q^n(h-c)}{h}\le l<q^n+\frac{c+q^n(h-c)}{h},\gcd(p,l\cdot v)=1,l^\sharp_N>d\right\}
		\end{align*}
		to 
		\begin{align*}
			\left\{l\in \Z^k: 1+\frac{(h-c)(q^n-1)}{Nh}\le l<1+\frac{(h-c)(q^n-1)}{Nh}+(N-d)\frac{q^n-1}{N}, \gcd(p,l\cdot v)=1\right\}.
		\end{align*}
	\end{cor}
	
	\begin{proof}
		The result follows directly from Lemma \ref{lem.ferrero-greenberg-map}, by applying it to the product $\Phi^n_{(h-c_1)/h}(d_1)\times\cdots\times \Phi^n_{(h-c_k)/h}(d_k)$.
	\end{proof}

	\begin{proof}[Proof of Proposition \ref{prop.ferrero-greenberg}]
		By the above results, we can carry out the following manipulations
		\begin{align*}
			&(-1)^{k-1} L_{p,V,x}'(0,\chi\omega_F)\\
			=&\lim_{n\to\infty}\sum_{\substack{0\le l_1,\cdots,l_k< q^n\\ \gcd(p,x+l\cdot v) = 1}} \sum_{0\le d<l^\flat_N} \chi(x+d\cdot v) \log_p\nm(x+l\cdot v)\\
			=&k\log_pN \lim_{n\to\infty}\sum_{\substack{0\le l_1,\cdots,l_k< q^n\\ \gcd(p,x+l\cdot v) = 1}} \sum_{0\le d<l^\flat_N} \chi(x+d\cdot v)
			+\lim_{n\to\infty}\sum_{1\le d\le N} \chi(x+(d-1)\cdot v)
			\sum_{\substack{\frac{c+q^n(h-c)}{h}\le l<q^n+\frac{c+q^n(h-c)}{h}\\ \gcd(p,l\cdot v)=1, l^\sharp_N>d}} \log_p(\nm(\iota(l)\cdot v))\\
			=&(-1)^k k\log_pN\cdot L_{p,V,x}(0,\chi\omega_F)
			+\lim_{n\to\infty}\sum_{1\le d\le N} \chi(x+(d-1)\cdot v)
			\sum_{\substack{1+\frac{(h-c)(q^n-1)}{Nh}\le l<1+\frac{(h-c)(q^n-1)}{Nh}+(N-d)\frac{q^n-1}{N}\\ \gcd(p,l\cdot v)=1}} \log_p\nm(l\cdot v)\\
			=&(-1)^k k\log_pN\cdot L_{p,V,x}(0,\chi\omega_F)\\
			&+\lim_{n\to\infty}\sum_{1\le d\le N} \chi(x+(d-1)\cdot v)
			\sum_{S\subseteq\{1,\cdots,k\}} (-1)^{|S|}
			\log_p\Gamma_{F,p,V}\left(\sum_{i\in S}\frac{c_i-h+Nh}{Nh}v_i+\sum_{j\notin S}\frac{c_j-h+d_jh}{Nh}v_j\right)\\
			=&(-1)^k k\log_pN\cdot L_{p,V,x}(0,\chi\omega_F) + 
			\sum_{0\le d< N} \chi(x+d\cdot v)
			\log_p\Gamma_{F,p,V}\left(\frac{c\cdot v}{Nh}+\frac{d\cdot v}{N}\right).
		\end{align*}
		Here in the second equality we used the congruence $N\iota(m)\equiv m\bmod q^n$, in the penultimate equality the inclusion-exclusion principle, and in the last equality the nontriviality of $\chi$.
	\end{proof}
	
	\begin{proof}[Proof of Corollary \ref{cor.ferrero-greenberg-formula}]
		Using Proposition \ref{prop.ferrero-greenberg}, we find
		\begin{align*}
			L'_{F,p}(0,\chi\omega_F) =& \sum_{1\le i\le h}\sum_V\sum_{x\in P(V)\cap \got{a}_i^{-1}} \chi(\got{a}_i)L'_{p,V,x}(0,\chi\omega_F)
			- \sum_{1\le i\le h}\sum_V\sum_{x\in P(V)\cap \got{a}_i^{-1}} \chi(\got{a}_i)\log_p\nm(\got{a}_i)L_{p,V,x}(0,\chi\omega_F)\\
			=&(-1)^{k-1}\sum_{1\le i\le h}\sum_V \sum_{x\in P(V)\cap \got{a}_i^{-1}}\chi(\got{a}_i(x+d\cdot v))\log_p\Gamma_{F,p,V}\left(\frac{x+d\cdot v}{N}\right)\\
			&-\sum_{1\le i\le h}\sum_V\sum_{x\in P(V)\cap \got{a}_i^{-1}} \chi(\got{a}_i) \left(\log_p\nm(\got{a}_i)+k\log_p N\right) L_{p,V,x}(0,\chi\omega_F).
		\end{align*}
		As $p$ is inert and $\chi(p) = 1$, a counting argument shows $\sum_{x\in P(V)\cap \got{a}_i^{-1}} L_{p,V,x}(0,\chi\omega_F)=0$ (see Appendix \ref{subsection.A.p-adic}, especially identity \eqref{equation.interpolation-x}). The proof is thus concluded.
	\end{proof}

	\appendix

	\section{Values at $s=0$}
	\label{section.appendix}
	
	Keeping the notation and assumptions from §\ref{section.ferrero-greenberg}, in this appendix we provide explicit formulas for the special values $L_{V,x}(0,\chi)$ and $L_{p,V,x}(0,\chi\omega_F)$, where $p$ is assumed to be inert in $F$ and $\cali{O}_{V,p}=\cali{O}_p$. The former is essentially due to Shintani \cite{Sh76}.
	
	\subsection{The complex formula}
	
	Given a triple $(V,x,\chi)$, define the complex Hecke-Shintani function
	\begin{align*}
		L_{V,x}(s,\chi) = \sum_{l\ge 0} \frac{\chi(x+l\cdot v)}{\nm(x+l\cdot v)^s}.
	\end{align*}
	For $y \in F$ write $y^{(i)}$ the image of $y$ under the $i$-th embedding $\sigma_i: F\to \R$. Consider the function on $\R_+^k$:
	\begin{align*}
		G_{V,x,\chi}(u_1,\cdots,u_k) = \sum_{0\le d_1,\cdots,d_k<N} \chi(x+d\cdot v) \prod_{1\le i\le k}\frac{e^{-(x_i+d_i)\sum_{1\le j\le k}v_i^{(j)}u_j}}{1-e^{-N\sum_{1\le j\le k}v_i^{(j)}u_j}} = \sum_{l_1,\cdots,l_k\ge 0} \frac{B_{l_1+1,\cdots,l_k+1}(\chi)}{\prod_{1\le i\le k}(l_i+1)!}u^l.
	\end{align*}
	Using Euler's method, one can show that (\textit{cf.~}\cite[§2.4]{Hida}):
	\begin{align*}
		L_{V,x}(0,\chi) = B_{1,\cdots,1}(\chi).
	\end{align*}
	To compute the constant term of $G_{V,x,\chi}(u)$, we use the orthogonality of $\chi$ in the same manner as §\ref{section.period2}:
	\begin{align*}
		G_{V,x,\chi}(u)&=\sum_{0\le d<N}\chi(x+d\cdot v) \prod_{1\le i\le k} \frac{e^{-(x_i+d_i)\sum_{1\le j\le k}v_i^{(j)}u_j}-1}{1-e^{-N\sum_{1\le j\le k}v_i^{(j)}u_j}}=(-1)^k\sum_{0\le d<N}\chi(x+d\cdot v)\prod_{1\le i\le k}\frac{x_i+d_i}{N} + O(u).
	\end{align*}
	Hence we recover the formula
	\begin{align}
		L_{V,x}(0,\chi) = \frac{(-1)^k}{N^k}\sum_{0\le d<N}\chi(x+d\cdot v)\prod_{1\le i\le k}d_i.
	\end{align}
	
	\subsection{The $p$-adic formula}
	\label{subsection.A.p-adic}
	
	Let $V$ be a generator set in the Shintani decomposition and $x\in P(V)\cap \got{a}_i^{-1}$ for some $i$. By our assumptions on $\got{a}_i$ and $V$, we have
	\begin{align*}\textstyle
		(\got{a}_i^{-1}/\sum_{1\le i\le k}\Z v_i)\otimes \Z_p = \cali{O}_p/\sum_{1\le i\le k}\Z_p v_i=0.
	\end{align*}
	So the finite group $(\got{a}_i^{-1}/\sum_{1\le i\le k}\Z v_i)$ is $p$-divisible. As $P(V)\cap \got{a}_i^{-1}$ is in set bijection with $\got{a}_i^{-1}/\sum_{1\le i\le k}\Z v_i$, we define $\tau_p$ to be the automorphism of $P(V)\cap \got{a}_i^{-1}$ corresponding to the $p$-multiplication of $\got{a}_i^{-1}/\sum_{1\le i\le k}\Z v_i$. Below we prove the interpolation formula
	\begin{align}
		\label{equation.interpolation-0}
		L_{p,V,x}(0,\chi\omega_F) = L_{V,x}(0,\chi) - \chi(p)L_{V,\tau_p^{-1}x}(0,\chi).
	\end{align}
	Granting \eqref{equation.interpolation-0}, it follows that
	\begin{align}
		\label{equation.interpolation-x}
		\sum_{x\in P(V)\cap \got{a}_i^{-1}} L_{p,V,x}(0,\chi\omega_F) = (1-\chi(p))\sum_{x\in P(V)\cap \got{a}_i^{-1}}L_{V,x}(0,\chi)
	\end{align}
	for any representative $\got{a}_i$ of $\cl_+(1)$, whereby $L_{p,F}(0,\chi\omega_F) = (1-\chi(p))L_F(0,\chi)$.
	
	\vspace{3mm}
	
	To begin our proof, applying Lemma \ref{lem.p-analyse} and Corollary \ref{cor.big-ferrero-greenberg-map} to the sum expression of $L_{p,V,x}(s,\chi\omega_F)$ for $x = \frac{c\cdot v}{h}$, we find
	\begin{align*}
		L_{p,V,x}(0,\chi\omega_F) = (-1)^k\sum_{1\le d_1,\cdots,d_k\le N}\chi(x+(d-1)\cdot v) \lim_{n\to \infty} 
		\sum_{\substack{1+\frac{(h-c)(q^n-1)}{Nh}\le l<1+\frac{(h-c)(q^n-1)}{Nh}+(N-d)\frac{q^n-1}{N}\\ \gcd(p,l\cdot v)=1}} 1.
	\end{align*}
	Next, note that
	\begin{align*}
		&\#\left\{1+\frac{(h-c)(q^n-1)}{Nh}\le l<1+\frac{(h-c)(q^n-1)}{Nh}+(N-d)\frac{q^n-1}{N}, \gcd(p,l\cdot v)=1\right\}\\
		=&\#\left\{1+\frac{(h-c)(q^n-1)}{Nh}\le l<1+\frac{(h-c)(q^n-1)}{Nh}+(N-d)\frac{q^n-1}{N}\right\} \\
		&-\#\left\{\frac{(h-c)q^n/p}{Nh}<l<\frac{(h-c)q^n/p}{Nh}+\frac{(N-d)q^n/p}{N}\right\}\\
		=&\prod_{1\le i\le k}\frac{(N-d_i)(q^n-1)}{N} - \prod_{1\le i\le k}\left[\frac{(N-d_i)q^n/p}{N}
		+\frac{(\frac{d_ih-h+c_i}{p})^*_{Nh}-(\frac{c_i-h}{p})^\sharp_{Nh}}{Nh}\right]\\
		\to&\prod_{1\le i\le k}\frac{d_i-N}{N} - \prod_{1\le i\le k} \frac{(\frac{d_ih-h+c_i}{p})^*_{Nh}-(\frac{c_i-h}{p})^\sharp_{Nh}}{Nh} \quad(n\to\infty),
	\end{align*}
	where $*$ is $\flat$ if $c_i<h$, and $*$ is $\sharp$ if $c_i= h$. Hence,
	\begin{align*}
		L_{p,V,x}(0,\chi\omega_F) &= \frac{(-1)^k}{N^k}\sum_{1\le d\le N} \chi(x+(d-1)\cdot v)\left[\prod_{1\le i\le k}d_i - \frac{1}{h^k}\prod_{1\le i\le k}\left(\frac{d_ih-h+c_i}{p}\right)^*_{Nh}\right]\\
		&= \frac{(-1)^k}{N^k}\sum_{0\le d<N} \chi(x+d\cdot v)\left[\prod_{1\le i\le k}d_i - \frac{1}{h^k}\prod_{1\le i\le k}\left(\frac{d_ih+c_i}{p}\right)^*_{Nh}\right].
	\end{align*}
	The latter sum can be further deformed:
	\begin{align*}
		-\frac{(-1)^k}{(Nh)^k}\sum_{0\le d<N}\chi(x+d\cdot v) \prod_{1\le i\le k} \left(\frac{d_ih+c_i}{p}\right)^*_{Nh}
		&=-\frac{(-1)^k}{(Nh)^k}\sum_{0\le d<N}\chi(x+pd\cdot v)\prod_{1\le i\le k}
		\left(d_i h + \frac{c_i}{p}\right)^*_{Nh}.
	\end{align*}
	Via an elementary argument, it can be shown that
	\begin{align*}
		\left(d_ih +\frac{c_i}{p}\right)^*_{Nh}=
		\begin{cases}
			d_ih + \left(\frac{c_i}{p}\right)^\flat_{Nh} - Nh\1_{d_i+\frac{1}{h}\left(\frac{c_i}{p}\right)^\flat_{Nh}\ge N} & \text{if }*=\flat,\text{ i.e., }c_i<h;\\
			d_ih + \left(\frac{c_i}{p}\right)^\flat_{Nh} - Nh\1_{d_i+\frac{1}{h}\left(\frac{c_i}{p}\right)^\flat_{Nh}> N} & \text{if }*=\sharp,\text{ i.e., }c_i=h.
		\end{cases}
	\end{align*}
	In the former case, we can rewrite it as
	\begin{align*}
		d_ih +\left(\frac{c_i}{p}\right)^\flat_{Nh} - Nh\1_{d_i+\frac{1}{h}\left[\left(\frac{c_i}{p}\right)^\flat_{Nh} - \left(\frac{c_i}{p}\right)^\flat_{h}\right]
			\ge N}
		&= h \left(d_i+\frac{1}{h}\left(\frac{c_i}{p}\right)^\flat_{Nh} -\frac{1}{h}\left(\frac{c_i}{p}\right)^\flat_{h}\right)^\flat_N + \left(\frac{c_i}{p}\right)^\flat_h\\
		&= h \left(d_i+\frac{1}{h}\left[\frac{c_i}{p} -\left(\frac{c_i}{p}\right)^\flat_{h}\right]\right)^\flat_N + \left(\frac{c_i}{p}\right)^\flat_h.
	\end{align*}
	In the latter case we have a similar recast:
	\begin{align*}
		d_ih +\left(\frac{c_i}{p}\right)^\flat_{Nh} - Nh\1_{d_i+\frac{1}{h}\left[\left(\frac{c_i}{p}\right)^\flat_{Nh} - \left(\frac{c_i}{p}\right)^\sharp_{h}\right]\ge N}
		= h \left(d_i + \frac{1}{h}
		\left[\frac{c_i}{p} -\left(\frac{c_i}{p}\right)^\sharp_{h}\right]\right)^\flat_N + \left(\frac{c_i}{p}\right)^\sharp_h.
	\end{align*}
	Therefore,
	\begin{align*}
		&-\frac{(-1)^k}{(Nh)^k}\sum_{0\le d<N}\chi(x+d\cdot v) \prod_{1\le i\le k} \left(\frac{d_ih+c_i}{p}\right)^\flat_{Nh}\\
		=& -\frac{(-1)^k}{N^k}
		\sum_{0\le d<N} \chi(x+pd\cdot v)\prod_{1\le i\le k}\left(d_i+\frac{1}{h}\left[\frac{c_i}{p}-\left(\frac{c_i}{p}\right)^*_h\right]\right)^\flat_N\\
		=&-\frac{(-1)^k}{N^k}\sum_{0\le d<N}\chi\left(\frac{c}{h}\cdot v +p\sum_{1\le i\le k}\left(d_i-\frac{1}{h}\left[\frac{c_i}{p}-\left(\frac{c_i}{p}\right)^*_h\right]\right)v_i\right)\prod_{1\le i\le k} d_i\\
		=&-\frac{(-1)^k\chi(p)}{N^k}\sum_{0\le d<N}\chi\left(\frac{\sum_{1\le i\le k}(c_i/p)^*_h v_i}{h} +d\cdot v\right)\prod_{1\le i\le k} d_i
	\end{align*}
	The last formula is no other than $-\chi(p)L_{V,\tau_p^{-1}x}(0,\chi)$, and our proof of \eqref{equation.interpolation-0} is thus complete.

	\section{Numerical computation of Iwasawa invariants over $\Q(\sqrt{5})$}
	
	\label{section.appendix-b}
	This appendix is dedicated to explaining how one may use the results from the main article to compute analytic Iwasawa $\lambda$- and $\mu$-invariants of an abelian extension of a given totally real field $F$ whose conductor divides a Cassou-Nogu\`es ideal; or more exactly that of the subfields cut out by characters of the corresponding Galois group. For simplicity we will only treat the case when $F=\Q(\sqrt{5})$, the Cassou-Nogu\`es ideal is a prime and the character is quadratic.
	
	\subsection{Setup}
	\label{subsection.appendix.setup}
	Throughout we let $F=\Q(\sqrt{5})$, $\cali{O}=\Z[\frac{1+\sqrt{5}}{2}]$ be the ring of integers in $F$, $p$ be a rational prime and $\cali{L}$ a Cassou-Nogu\`es prime of $F$, namely $\cali{O}/\cali{L} = \Z/\ell$, with $\ell = \nm(\cali{L})$. We have an identification $\cl_+(\cali{L}) = (\Z/\ell)^\times/\chx{\varepsilon}$ with $\varepsilon = \frac{3+\sqrt{5}}{2}$, since the narrow class group of $F$ is trivial. Therefore, a character on $\cl_+(\cali{L})$ can be identified as a Dirichlet character on $(\Z/\ell)^\times$. By the functional equations of the complex $L$-functions, the attached $p$-adic $L$-function $L_p(s,\chi\omega)$ is nonzero if and only if $\chi$ is totally odd. In simple terms, this requires that the character is such that
	\begin{align*}
		\chi(-1) = 1\quad \text{and}\quad \chi\left(\frac{1+\sqrt{5}}{2}\right) = -1.
	\end{align*}
	Eventually we will only consider the case when $\chi$ is quadratic. Under this condition, for $\ell<1000$, $\ell$ could only be one of the primes below:
	\begin{align*}
		41, 61, 109, 149, 241, 269, 281, 389, 409, 421, 449, 569, 601, 641, 661, 701, 821, 829, 881, 929.
	\end{align*}
	
	The major computational tool we will use is the period formula from Theorem 4.4 of the main article. Namely, for $0\le l_1,l_2<p^n$,
	\begin{align}
		\label{equation.period-formula}
		\mu(1+l_1+l_2\varepsilon + p^n\cali{O}_p) = \frac{1}{\ell^2}\sum_{1\le d_1,d_2<\ell} \chi(1+l_1+p^nd_1+(l_2+p^nd_2)\varepsilon)d_1d_2,
	\end{align}
	where $\mu$ is the attached $p$-adic measure of $L_p(s,\chi\omega)$. Recall also we have the integral representation
	\begin{align}
		\label{equation.integral-representation}
		L_p(s,\chi\omega) = \int_{\cali{O}_p^\times} \chx{\nm\alpha}^{-s}\mu(\alpha);
	\end{align}
	or if we let $s = -k$ where $k\in \Z_{>0}$,
	\begin{align*}
		L(-k,\chi\omega^k) = \int_{\cali{O}_p}\nm\alpha^k \mu(\alpha),
	\end{align*}
	with $L(s,\chi\omega^k)$ being the complex Hecke $L$-function.
	
	\vspace{3mm}
	
	As an exercise, one can compute $\mu(1+\cali{O}_p)$ for the Cassou-Nogu\`es primes $\ell = 11, 19, 29, 31, 41$ and the attached quadratic character. Among them only when $\ell = 41$ the value is nonzero ($=2$), and this echoes with our note that the smallest interesting $\ell$ is $41$.
	
	\subsection{Approximate Iwasawa functions}
	\label{subsection.appendix.approximate}
	In the following assume $p\ne 2$. Fix now a topological generator $\gamma$ of $1+p\Z_p$. With the same setting as in §\ref{subsection.appendix.setup}, the Iwasawa function $G_{\chi\omega}(T)\in \cali{O}_\chi[[T]]$ is defined by (\textit{cf.}~\cite[p.~494, (1.3)]{Wi90})
	\begin{align*}
		G_{\chi\omega}(\gamma^{1-s}-1) = L_p(s,\chi\omega);
	\end{align*}
	here $\cali{O}_\chi$ is $\cali{O}$ adjoining values of $\chi$. Using \eqref{equation.integral-representation} and the fact that a power series in $\cali{O}_\chi[[T]]$ has finitely many zeroes, we deduce that
	\begin{align*}
		G_{\chi\omega}(\gamma(1+T)-1) = \int_{\cali{O}_p^\times} (1+T)^{\log_\gamma(\nm\alpha)}\mu(\alpha),
	\end{align*}
	where $\log_\gamma x = \log_p x/\log_p \gamma$, with $\log_p$ being the $p$-adic logarithm. From this one further obtains the congruence
	\begin{align}
		\label{equation.congruence}
		\begin{split}
			&G_{\chi\omega}(\gamma(1+T)-1)\\ 
			\equiv & \sum_{0\le m<p^n} (1+T)^m \sum_{\substack{0\le l_1,l_2<p^{n+1}\\ \gcd(p,1+l_1+l_2\varepsilon)=1\\ \chx{\nm(1+l_1+l_2\varepsilon)}\equiv \gamma^m \bmod p^{n+1}}} \mu(1+l_1+l_2\varepsilon+p^{n+1}\cali{O}_p) \pmod {(1+T)^{p^n}-1}.
		\end{split}
	\end{align}
	It can be shown that the $\mu$- and $\lambda$-invariants of $G_{\chi\omega}(\gamma(1+T)-1)$ coincide with those of $G_{\chi\omega}(T)$, so we do not distinguish them. Assuming $\mu(G_{\chi\omega})=0$, we see that the $\lambda(G_{\chi\omega})$ can be pinned down by looking at the smallest $r\ge 0$ such that the coefficient of $T^r$ on the right is a $p$-unit, if such coefficients exist.
	
	\vspace{3mm}
	
	We describe the steps to execute the computations, given $(p,n,\cali{L})$:
	\begin{enumerate}
		\item Choose a set of representatives $\{a_0=1, a_1,\cdots,a_{p^n-1}\}$ of $(1+p\Z_p)/(1+p^{n+1}\Z_p)$, such that $a_1^j\equiv a_j\bmod p^{n+1}$ for $0\le j< p^n$.
		
		\item Compute the table of $\nm(1+l_1+l_2\epsilon)=(1+l_1)^2+3(1+l_1)l_2+l_2^2 \bmod p^{n+1}$ for $0\le l_1,l_2<p^{n+1}$; dispose of these entries with $p\mid (1+l_1)^2+3(1+l_1)l_2+l_2^2$.
		
		\item Partition the remaining pairs $(l_1,l_2)$ in $p^n$ groups $\{S_{i}\}_{0\le i< p^n}$, according to $\chx{\nm(1+l_1+l_2\epsilon)} \equiv a_i\bmod p^{n+1}$.
		
		\item Compute the table of values $\mu(1+l_1+l_2\epsilon+p^{n+1}\cali{O}_p)$ for $0\le l_1,l_2<p^{n+1}$; note that in practice this only takes $O(\ell^4)$ time, since $\mu(1+l_1+l_2\epsilon+p^{n+1}\cali{O}_p)$ is $\ell$-periodic in $l_1,l_2$.
		
		\item For each group $S_i$, compute the sum $A_i = \sum_{(l_1,l_2)\in S_i}\mu(1+l_1+l_2\epsilon+p^{n+1}\cali{O}_p)$. Note that \eqref{equation.congruence} implies that
		\begin{align}
			\label{equation.approx-iwasawa-functions}
			G_{\chi\omega}(\gamma(1+T)-1) \equiv \sum_{0\le i<p^n} A_i (1+T)^i \pmod {(1+T)^{p^n}-1}.
		\end{align}
	\end{enumerate}
	
	\subsection{Numerical examples}
	\label{subsection.appendix.numerical}
	First let $p=3$ and $n=1$. We can first carry out steps 1,2,3 above, since they do not require the knowledge of $\cali{L}$. We have $(\Z/9)^\times\simeq \{\pm 1\}\times \{1+3 = 4, 1+3\times 2 = 7, 1\}$, so we take $a_1=4, a_2=7$. In the following table, we record the partition of pairs $(l_1,l_2)$ for $0\le l_1,l_2<9$ into groups $S_0,S_1,S_2$ (an empty cell means the corresponding $(l_1,l_2)$ is disposed of):
	
	\begin{align*}
		\begin{array}{| c | c | c | c | c | c | c | c | c | c |}
			\hline
			& l_2 = 0 & 1 & 2 & 3 & 4 & 5 & 6 & 7 & 8\\ \hline
			l_1 = 0 & 0 & 1 & 2 & 0 & 2 & 1 & 0 & 0 & 0\\ \hline
			1 & 1 & 2 & 2 & 1 & 0 & 1 & 1 & 1 & 0\\ \hline
			2 & & 0 & 1 & & 2 & 2 & & 1 & 0\\ \hline
			3 & 2 & 2 & 0 & 2 & 0 & 2 & 2 & 1 & 1\\ \hline
			4 & 2 & 1 & 1 & 2 & 2 & 0 & 2 & 0 & 2\\ \hline
			5 &  & 0 & 1 &  & 2 & 2 &  & 1 & 0\\ \hline
			6 & 1 & 0 & 1 & 1 & 1 & 0 & 1 & 2 & 2\\ \hline
			7 & 0 & 0 & 0 & 0 & 1 & 2 & 0 & 2 & 1\\ \hline
			8 & & 0 & 1 & & 2 & 2 & & 1 & 0 \\ \hline
		\end{array}
	\end{align*}
	As such, one can compute $A_0,A_1,A_2$ for varying Cassou-Nogu\`es primes $\cali{L}$. Below, given a split rational prime $\ell$, we always pick the prime $\mathcal{L}$ above $\ell$ to be the one such that $\varepsilon^\flat_\cali{L}<\varepsilon^\flat_{\cali{L}^\sigma}$, where $\sigma\in\gal(\Q(\sqrt{5})/\Q)$ is the nontrivial element. (Recall for $a\in \cali{O}/\cali{L}=\Z/\ell$, $a^\flat_{\cali{L}}$ denotes the integer in $[0,\ell)$ with $a\equiv a^\flat_{\cali{L}}\bmod \cali{L}$.) For example, with $\ell=41$, $\cali{L} = (2+3\sqrt{5})$ since $\varepsilon^\flat_{(2+3\sqrt{5})} = 8<\varepsilon^\flat_{(2-3\sqrt{5})} = 36$. As expected, all of the three coefficients vanish for $\ell<41$. For interesting $\ell$'s we find the following table:
	\begin{align*}
		\begin{array}{| c | c | c | c | c | c |}
			\hline
			\ell & A_0 & A_1 & A_2 & G_{\chi\omega}(T) \bmod (1+T)^3-1 & \lambda_3\\ \hline
			41 & -4 & 8 & 0 & 4 + 8T &0\\ \hline
			61 & 4 & -12 & 8 & 4T+8T^2 &1\\ \hline
			109 & 8 & 8 & -16 & -24T-16T^2 & 2\\ \hline
			149 & -20 & 8 & 16 & 4 + 40T + 16T^2 & 0\\ \hline
			241 & -4 & -12 & 16 & 20T+16T^2 & 1\\ \hline
			269 & 20 & 12 & -28 & 4 - 44T -28T^2 & 0\\ \hline
			281 & 0 & -8 & 20 & 12 + 32T +20T^2 & 1\\ \hline
			389 & 16 & -32 & 20 & 4 + 8T + 20 T^2 & 0\\ \hline
			409 & -24 & 20 & 4 & 28T + 4T^2 & 1\\ \hline
			421 & 12 & -28 & 16 & 4T + 16T^2 & 1\\ \hline
			449 & 20 & 12 & -20 & 12 - 28T - 20T^2 & 1\\ \hline
			569 & 8 & -24 & 28 & 12 + 32T + 28T^2 & 1\\ \hline
			601 & -4 & -20 & 24 & 28T + 24T^2 & 1\\ \hline
			641 & 8 & -16 & 28 & 20 + 40T + 28T^2 & 0\\ \hline
			661 & -32 & 36 & -4 & 28T - 4T^2 & 1\\ \hline
			701 & 36 & 12 & -36 & 12 - 60T -36T^2 & ?\\ \hline
			821 & 16 & -32 & 28 & 12 +24T +28T^2 & 2\\ \hline
			829 & 8 & 32 & -40 & -48T-40T^2 & 2\\ \hline
			881 & 44 & -4 & -20 & 20 -44T -20T^2 & 0\\ \hline
			929 & 8 & -32 & 36 & 12 + 40T + 36T^2 & 1\\ \hline
		\end{array}
	\end{align*}
	In particular, we have numerically verified $\mu_3=0$ in these cases except for $\ell=701$. Similarly we do this for $n=2$, and obtain a table
	\begin{align*}
		\begin{array}{| c | c | c | c | c | c | c | c | c | c |}
			\hline
			\ell & A_0 & A_1 & A_2 & A_3 & A_4 & A_5 & A_6 & A_7 & A_8 \\ \hline
			41 & 4 & 8 & -4 & 8 & -16 & -8 & -16 & 16 & 12\\ \hline
			61 & 4 & 16 & -4 & -12 & -4 & -8 & 12 & -24 & 20\\ \hline
			109 & 24 & 16 & 16 & -16 & -24 & 0 & 0 & 16 & -32\\ \hline
			149 & -12 & 8 & -12 & 24 & -16 & -8 & -32 & 16 & 36\\ \hline
			241 & -36 & -12 & -4 & 12 & -24 & 48 & 20 & 24 & -28\\ \hline
			269 & -12 & -20 & -52 & 28 & 52 & -12 & 4 & -20 & 36\\ \hline
			281 & -8 & 20 & -40 & -24 & -36 & 44 & 32 & 8 & 16\\ \hline
			389 & -24 & 36 & -24 & -16 & -68 & 36 & 56 & 0 & 8\\ \hline
			409 & -36 & 52 & 40 & 12 & -44 & -52 & 0 & 12 & 16\\ \hline
			421 & 52 & -72 & 68 & -4 & 48 & -16 & -36 & -4 & -36\\ \hline
			449 & 12 & -20 & 12 & -52 & -28 & -24 & 60 & 60 & -8\\ \hline
			569 & 64 & -16 & 0 & -16 & 28 & -44 & -40 & -36 & 72\\ \hline
			601 & 36 & 48 & -44 & -52 & -28 & -8 & 12 & -40 & 76\\ \hline
			641 & -24 & -52 & 64 & 56 & 8 & 36 & -24 & 28 & -72\\ \hline
			661 & -24 & -12 & 32 & -76 & 60 & 12 & 68 & -12 & -48\\ \hline
			701 & -44 & -4 & -76 & 20 & 60 & -20 & 60 & -44 & 60\\ \hline
			821 & -40 & 68 & -64 & 16 & -108 & 60 & 40 & 8 & 32\\ \hline
			829 & -32 & -72 & -4 & -24 & 64 & -76 & 64 & 40 & 40\\ \hline
			881 & 28 & -4 & 44 & -44 & -52 & -72 & 60 & 52 & 8\\ \hline
			929 & -8 & 44 & -56 & -72 & -60 & 108 & 88 & -16 & -16\\ \hline
		\end{array}
	\end{align*}
	We record the approximate Iwasawa functions modulo $(1+T)^9-1$ separately:
	\begin{align*}
		\begin{array}{|c|c|}
			\hline
			\ell = 41 & 4+ 32T+ 276T^2+ 776T^3+ 1104T^4+ 904T^5+ 432T^6+ 112T^7+ 12T^8\\ \hline
			61 &  -20T+ 92T^2+ 412T^3+ 696T^4+ 680T^5+ 404T^6+ 136T^7+ 20T^8\\ \hline
			109 & -240T+ -736T^2 -1344T^3 -1704T^4 -1456T^5 -784T^6 -240T^7 -32T^8\\ \hline
			149 & 4+ 160T+ 748T^2+ 1816T^3+ 2544T^4+ 2152T^5+ 1088T^6+ 304T^7+ 36T^8\\ \hline
			241 & 224T+ 388T^2+ 68T^3 -604T^4 -896T^5 -596T^6 -200T^7 -28T^8\\ \hline
			269 & 4+ 280T+ 872T^2+ 1512T^3+ 1872T^4+ 1608T^5+ 872T^6+ 268T^7+ 36T^8\\ \hline
			281 & 12+ 320T+ 1208T^2+ 2088T^3+ 2064T^4+ 1300T^5+ 536T^6+ 136T^7+ 16T^8\\ \hline
			389 & 4+ 248T+ 944T^2+ 1640T^3+ 1512T^4+ 820T^5+ 280T^6+ 64T^7+ 8T^8\\ \hline
			409 & -56T -8T^2+ 632T^3+ 1236T^4+ 1096T^5+ 532T^6+ 140T^7+ 16T^8\\ \hline
			421 & -368T -1448T^2 -2848T^3 -3232T^4 -2332T^5 -1072T^6 -292T^7 -36T^8\\ \hline
			449 & 12+ 332T+ 1384T^2+ 2448T^3+ 2292T^4+ 1148T^5+ 256T^6 -4T^7 -8T^8\\ \hline
			569 & 12 -88T+ 340T^2+ 1628T^3+ 2988T^4+ 2992T^5+ 1724T^6+ 540T^7+ 72T^8\\ \hline
			601 & 52T+ 1020T^2+ 2852T^3+ 4032T^4+ 3480T^5+ 1860T^6+ 568T^7+ 76T^8\\ \hline
			641 & 20 -68T -1148T^2 -3084T^3 -4232T^4 -3552T^5 -1844T^6 -548T^7 -72T^8\\ \hline
			661 & 64T -292T^2 -1464T^3 -2640T^4 -2520T^5 -1360T^6 -396T^7 -48T^8\\ \hline
			701 & 12+ 576T+ 1800T^2+ 3080T^3+ 3520T^4+ 2776T^5+ 1432T^6+ 436T^7+ 60T^8\\ \hline
			821 & 12+ 408T+ 1600T^2+ 3056T^3+ 3312T^4+ 2260T^5+ 992T^6+ 264T^7+ 32T^8\\ \hline
			829 & 708T+ 2468T^2+ 4392T^3+ 4844T^4+ 3388T^5+ 1464T^6+ 360T^7+ 40T^8\\ \hline
			881 & 20+ 172T+ 1096T^2+ 2496T^3+ 2868T^4+ 1828T^5+ 648T^6+ 116T^7+ 8T^8\\ \hline
			929 & 12+ 304T+ 984T^2+ 1072T^3+ 120T^4 -596T^5 -472T^6 -144T^7 -16T^8\\ \hline
		\end{array}
	\end{align*}
	This shows that $\mu_3$ is also zero for $\ell=701$, in which case $\lambda_3=3$.
	
	\vspace{3mm}
	
	For $p=5$ and $n=1$, we record the values of $A_i$'s:
	\begin{align*}
		\begin{array}{| c | c | c | c | c | c | c | c |}
			\hline
			\ell & A_0 & A_1 & A_2 & A_3 & A_4 & G_{\chi\omega}(T)\bmod (1+T)^5-1 & \lambda_5\\ \hline
			41 & -4 & -4 & 0 & 12 & -4 & 16T+ 12T^2 -4T^3 -4T^4 & 1\\ \hline
			61 & -4 & 4 & 12 & -12 & 0 & -8T -24T^2 -12T^3 & 1\\ \hline
			109 & 0 & 8 & -4 & -16 & 16 & 4+ 16T +44T^2 +48T^3+ 16T^4 & 0\\ \hline
			149 & -20 & 20 & 4 & 4 & -4 & 4+ 24T -8T^2 -12T^3 -4T^4 & 0\\ \hline
			241 & -4 & -12 & 0 & 28 & -12 & 24T +12T^2 -20T^3 -12T^4 & 1\\ \hline
			269 & -8 & -28 & 24 & 4 & 12 & 4+ 80T+ 108T^2+ 52T^3+ 12T^4 & 0\\ \hline
			281 & 4 & 32 & -12 & -16 & -8 & -72T -108T^2 -48T^3 -8T^4 & 1\\ \hline
			389 & 16 & -8 & -40 & 32 & 4 & 4+ 24T+ 80T^2+ 48T^3+ 4T^4 & 0\\ \hline
			409 & 8 & 16 & 4 & -32 & 16 & 12 -8T +4T^2+ 32T^3+ 16T^4 & 0\\ \hline
			421 & -32 & 4 & -8 & 0 & 36 & 132T+ 208T^2+ 144T^3+ 36T^4 & 1\\ \hline
			449 & -36 & 20 & 20 & 4 & 4 & 12+ 88T+ 56T^2+ 20T^3+ 4T^4 & 0\\ \hline
			569 & 0 & -44 & 16 & 20 & 20 & 12+ 128T+ 196T^2+ 100T^3+ 20T^4 & 0\\ \hline
			601 & 40 & -16 & -8 & -16 & 0 & -80T+ -56T^2 -16T^3 & 2\\ \hline
			641 & -20 & -4 & 0 & 44 & -20 & 48T+ 12T^2 -36T^3 -20T^4 & 1\\ \hline
			661 & -28 & 20 & 36 & -36 & 8 & 16T -24T^2 -4T^3+ 8T^4 & 1\\ \hline
			701 & 48 & -48 & 0 & -8 & 8 & -40T+ 24T^2+ 24T^3+ 8T^4 & 2\\ \hline
			821 & -40 & -12 & -8 & 0 & 60 & 212T+ 352T^2+ 240T^3+ 60T^4 & 1\\ \hline
			829 & 52 & -8 & 12 & 4 & -48 & 12 -164T -264T^2 -188T^3 -48T^4 & 0\\ \hline
			881 & 12 & 56 & -20 & -24 & -24 & -152T -236T^2 -120T^3 -24T^4 & 1\\ \hline
			929 & 12 & 24 & 28 & 4 & -56 & 12 -132T -296T^2 -220T^3 -56T^4 & 0\\ \hline
		\end{array}
	\end{align*}
	For $p=5$ and $n=2$, we have the following table
	\begin{align*}
		\begin{array}{|c|c|}
			\hline
			\ell & A_i,\ 0\le i<25\\ \hline
			41 & 0, 12, 0, -16, -24, -16, 0, 4, 0 , 8 , -20 , 8 , 20 , 0 , 20 , 28 , -4 , -20 , 32 , -16 , 4 , -20 , -4 , -4 , 8\\ \hline
			61 & 16, 24, -12, 20, 28, 4, -32, 32, -12, -8, -48, -8, 8, -8, 4, 16, 16, -12, -24, -20, 8, 4, -4, 12, -4\\ \hline
			109 & 28, -4, 0, 0, -20, 4, -16, -40, 16, -28, -28, 0, 40, -40, 20, 12, 56, 12, -12, 12, -16, -28, -16, 20, 32\\ \hline
			149 & -12, 12, -12, -24, -28, 32, 40, 28, 0, 8, 8, -28, 20, -32, -36, 0, -28, -40, -8, 52, -48, 24, 8, 68, 0\\ \hline
			241 & -32, 4, 52, -16, 48, 68, -8, -60, 88, -36, -4, -52, -16, -12, 16, 0, 24, -8, -36, -52, -36, 20, 32, 4, 12\\ \hline
			269 & 4, -24, 8, -24, -20, -36, 44, 68, 44, -8, -12, -4, -44, 24, -32, -48, 8, -44, -40, -12, 84, -52, 36, 0, 84\\ \hline
			281 & 68, -8, 40, 64, -20, -60, 92, -36, 8, -52, -12, -16, 20, 12, 36, 0, -44, -64, -40, 16, 8, 8, 28, -60, 12\\ \hline
			389 & -36, -40, 40, 88, 60, 0, -12, -12, -60, 28, -48, -56, 8, -48, -44, 0, 88, -52, 56, 0, 100, 12, -24, -4, -40\\ \hline
			409 & -80, 4, 72, -112, 48, -4, 76, 20, 12, -24, -8, -36, 12, 52, 64, 52, -24, -32, -4, -16, 48, -4, -68, 20, -56\\ \hline
			421 & 64, 88, -4, -80, 96, -36, -52, -100, -24, 16, -32, 24, 56, 36, -32, -56, -48, -4, 4, -12, 28, -8, 44, 64, -32\\ \hline
			449 & 4, -32, 68, -32, -80, 16, -64, -96, 16, 56, -128, 56, -4, 52, 32, 20, -16, 4, -36, 0, 52, 76, 48, 4, -4\\ \hline
			569 & -16, -8, -36, 80, -20, -80, 12, -76, -96, 16, 84, -124, 56, -8, 68, 12, 20, -24, -12, -32, 0, 56, 96, 56, -12\\ \hline
			601 & -24, 40, 8, 52, -8, -60, -76, -56, 28, 44, 12, 16, -56, 20, 80, -28, 64, 100, -20, -80, 140, -60, -4, -96, -36\\ \hline
			641 & -56, -8, 4, 4, 40, -80, 28, 100, -28, 72, 120, -20, -84, 156, -64, 0, -44, -24, -24, 24, -4, 40, 4, -64, -92\\ \hline
			661 & 20, -112, 88, -48, -16, -156, -28, 16, -4, 40, 52, 56, -44, -100, -68, 28, 36, 16, 48, -20, 28, 68, -40, 68, 72\\ \hline
			701 & -60, -108, -56, 0, 4, -16, 60, -40, 56, 76, -8, 72, 100, 32, -120, 104, -44, -24, -132, -4, 28, -28, 20, 36, 52\\ \hline
			821 & 88, -36, 72, 104, -20, 80, 72, -4, -124, 88, -80, -24, -132, -28, 44, -20, 56, 64, 56, -56, -108, -80, -8, -8, 4\\ \hline
			829 & -84, -116, 0, 100, -132, 56, 40, 144, 60, -20, 52, -20, -68, -56, 36, 68, 68, -8, 0, 16, -40, 20, -56, -100, 52 \\ \hline
			881 & -116, 148, -60, 36, -132, 8, -44, 52, 0, 48, -4, -84, -120, -60, 32, 36, 16, 40, -112, 36, 88, 20, 68, 112, -8\\ \hline
			929 & -52, 112, -28, -92, -8, -84, -100, 40, 120, -132, 84, -36, 120, 24, 24, -56, -28, -68, 4, 100, 120, 76, -36, -52, -40\\ \hline
		\end{array}
	\end{align*}
	
	Again, we can see that $\mu_5=0$ in these cases. An extra feature is that all the coefficients $A_i$ computed above are divisible by $4$; this is strongly reminiscent of the Deligne-Ribet $2$-divisibility \cite[Theorem 8.11]{DR80} with $\Q(\sqrt{5})$ being non-exceptional.
	
	\bibliographystyle{alpha}
	\bibliography{sumExprTotReal_v5.1}

\begin{thebibliography}{DPV21b}

\bibitem[Bar78]{Ba78}
Daniel Barsky.
\newblock Fonctions zeta {$p$}-adiques d'une classe de rayon des corps de
  nombres totalement r\'{e}els.
\newblock In {\em Groupe d'{E}tude d'{A}nalyse {U}ltram\'{e}trique (5e
  ann\'{e}e: 1977/78)}, pages Exp. No. 16, 23. Secr\'{e}tariat Math., Paris,
  1978.

\bibitem[CN79a]{CN79b}
Pierrette Cassou-Nogu\`es.
\newblock Analogues {$p$}-adiques des fonctions {$\Gamma $}-multiples.
\newblock In {\em Journ\'{e}es {A}rithm\'{e}tiques de {L}uminy ({C}olloq.
  {I}nternat. {CNRS}, {C}entre {U}niv. {L}uminy, {L}uminy, 1978)}, volume~61 of
  {\em Ast\'{e}risque}, pages 43--55. Soc. Math. France, Paris, 1979.

\bibitem[CN79b]{CN79}
Pierrette Cassou-Nogu\`es.
\newblock Valeurs aux entiers n\'{e}gatifs des fonctions z\^{e}ta et fonctions
  z\^{e}ta {$p$}-adiques.
\newblock {\em Invent. Math.}, 51(1):29--59, 1979.

\bibitem[Col88]{Co88}
Pierre Colmez.
\newblock R\'{e}sidu en {$s=1$} des fonctions z\^{e}ta {$p$}-adiques.
\newblock {\em Invent. Math.}, 91(2):371--389, 1988.

\bibitem[Das08]{Da08}
Samit Dasgupta.
\newblock Shintani zeta functions and {G}ross-{S}tark units for totally real
  fields.
\newblock {\em Duke Math. J.}, 143(2):225--279, 2008.

\bibitem[DDP11]{DDP}
Samit Dasgupta, Henri Darmon, and Robert Pollack.
\newblock Hilbert modular forms and the {G}ross-{S}tark conjecture.
\newblock {\em Ann. of Math. (2)}, 174(1):439--484, 2011.

\bibitem[Del82]{De82}
Pierre Deligne.
\newblock Hodge cycles on abelian varieties.
\newblock In {\em Hodge cycles, motives, and {S}himura varieties}, pages
  9--100. Lecture Notes in Mathematics, Vol. 900, 1982.
\newblock Notes by J.~Milne.

\bibitem[Del06]{De06}
Daniel Delbourgo.
\newblock A {D}irichlet series expansion for the {$p$}-adic zeta-function.
\newblock {\em J. Aust. Math. Soc.}, 81(2):215--224, 2006.

\bibitem[Del09a]{De09}
Daniel Delbourgo.
\newblock The convergence of {E}uler products over {$p$}-adic number fields.
\newblock {\em Proc. Edinb. Math. Soc. (2)}, 52(3):583--606, 2009.

\bibitem[Del09b]{De09-2}
Daniel Delbourgo.
\newblock Zeta-functions through the 2-adic looking glass.
\newblock {\em Austral. Math. Soc. Gaz.}, 36(4):266--272, 2009.

\bibitem[DK20]{DK20}
Samit Dasgupta and Mahesh Kakde.
\newblock On the {B}rumer-{S}tark {C}onjecture, 2020.
\newblock Annals of Mathematics, to appear, arxiv:2010.00657.

\bibitem[DK21]{DK21}
Samit Dasgupta and Mahesh Kakde.
\newblock Brumer-{S}tark {U}nits and {H}ilbert's 12th {P}roblem, 2021.
\newblock Preprint, arxiv:2103.02516.

\bibitem[DKV18]{DKV}
Samit Dasgupta, Mahesh Kakde, and Kevin Ventullo.
\newblock On the {G}ross-{S}tark conjecture.
\newblock {\em Ann. of Math. (2)}, 188(3):833--870, 2018.

\bibitem[DPV21a]{DPV}
Henri Darmon, Alice Pozzi, and Jan Vonk.
\newblock Diagonal restrictions of {$p$}-adic {E}isenstein families.
\newblock {\em Math. Ann.}, 379(1-2):503--548, 2021.

\bibitem[DPV21b]{DPV2}
Henri Darmon, Alice Pozzi, and Jan Vonk.
\newblock The values of the {D}edekind-{R}ademacher cocycle at real
  multiplication points, 2021.
\newblock Journal of the European Mathematical Society, to appear,
  arxiv:2103.02490.

\bibitem[DR80]{DR80}
Pierre Deligne and Kenneth~A. Ribet.
\newblock Values of abelian {$L$}-functions at negative integers over totally
  real fields.
\newblock {\em Invent. Math.}, 59(3):227--286, 1980.

\bibitem[DyDF12]{DF12}
Francisco Diaz~y Diaz and Eduardo Friedman.
\newblock Colmez cones for fundamental units of totally real cubic fields.
\newblock {\em J. Number Theory}, 132(8):1653--1663, 2012.

\bibitem[Fer78]{Fe78}
Bruce Ferrero.
\newblock Iwasawa invariants of {A}belian number fields.
\newblock {\em Math. Ann.}, 234(1):9--24, 1978.

\bibitem[FG79]{FG78}
Bruce Ferrero and Ralph Greenberg.
\newblock On the behavior of {$p$}-adic {$L$}-functions at {$s=0$}.
\newblock {\em Invent. Math.}, 50(1):91--102, 1978/79.

\bibitem[FW79]{FW79}
Bruce Ferrero and Lawrence~C. Washington.
\newblock The {I}wasawa invariant {$\mu _{p}$} vanishes for abelian number
  fields.
\newblock {\em Ann. of Math. (2)}, 109(2):377--395, 1979.

\bibitem[GK79]{GK79}
Benedict~H. Gross and Neal Koblitz.
\newblock Gauss sums and the {$p$}-adic {$\Gamma $}-function.
\newblock {\em Ann. of Math. (2)}, 109(3):569--581, 1979.

\bibitem[Gro81]{Gr81}
Benedict~H. Gross.
\newblock {$p$}-adic {$L$}-series at {$s=0$}.
\newblock {\em J. Fac. Sci. Univ. Tokyo Sect. IA Math.}, 28(3):979--994 (1982),
  1981.

\bibitem[Hid93]{Hida}
Haruzo Hida.
\newblock {\em Elementary theory of {$L$}-functions and {E}isenstein series},
  volume~26 of {\em London Mathematical Society Student Texts}.
\newblock Cambridge University Press, Cambridge, 1993.

\bibitem[Iwa58]{Iw58}
Kenkichi Iwasawa.
\newblock On some invariants of cyclotomic fields.
\newblock {\em Amer. J. Math. 80 (1958), 773-783; erratum}, 81:280, 1958.

\bibitem[Kas05]{Kas05}
Tomokazu Kashio.
\newblock On a {$p$}-adic analogue of {S}hintani's formula.
\newblock {\em J. Math. Kyoto Univ.}, 45(1):99--128, 2005.

\bibitem[Kat81]{Ka81}
Nicholas~M. Katz.
\newblock Another look at {$p$}-adic {$L$}-functions for totally real fields.
\newblock {\em Math. Ann.}, 255(1):33--43, 1981.

\bibitem[KW21]{KW21}
Heiko Knospe and Lawrence~C. Washington.
\newblock Dirichlet series expansions of {$p$}-adic {$L$}-functions.
\newblock {\em Abh. Math. Semin. Univ. Hambg.}, 91(2):325--334, 2021.

\bibitem[LV22]{LV22}
Alan Lauder and Jan Vonk.
\newblock Computing {$p$}-adic {L}-functions of totally real fields.
\newblock {\em Math. Comp.}, 91(334):921--942, 2022.

\bibitem[Mor75]{Mo75}
Yasuo Morita.
\newblock A {$p$}-adic analogue of the {$\Gamma $}-function.
\newblock {\em J. Fac. Sci. Univ. Tokyo Sect. IA Math.}, 22(2):255--266, 1975.

\bibitem[Rib79]{Ri79}
Kenneth~A. Ribet.
\newblock Report on {$p$}-adic {$L$}-functions over totally real fields.
\newblock In {\em Journ\'{e}es {A}rithm\'{e}tiques de {L}uminy ({C}olloq.
  {I}nternat. {CNRS}, {C}entre {U}niv. {L}uminy, {L}uminy, 1978)}, volume~61 of
  {\em Ast\'{e}risque}, pages 177--192. Soc. Math. France, Paris, 1979.

\bibitem[Rob15]{Ro15}
Xavier-Fran\c{c}ois Roblot.
\newblock Computing {$p$}-adic {$L$}-functions of totally real number fields.
\newblock {\em Math. Comp.}, 84(292):831--874, 2015.

\bibitem[Shi76]{Sh76}
Takuro Shintani.
\newblock On evaluation of zeta functions of totally real algebraic number
  fields at non-positive integers.
\newblock {\em J. Fac. Sci. Univ. Tokyo Sect. IA Math.}, 23(2):393--417, 1976.

\bibitem[Wil90]{Wi90}
A.~Wiles.
\newblock The {I}wasawa conjecture for totally real fields.
\newblock {\em Ann. of Math. (2)}, 131(3):493--540, 1990.

\bibitem[Yam10]{Ya10}
Shuji Yamamoto.
\newblock On {S}hintani's ray class invariant for totally real number fields.
\newblock {\em Math. Ann.}, 346(2):449--476, 2010.

\bibitem[Zha22]{Zh22}
Luochen Zhao.
\newblock Sum expressions for {$p$}-adic {K}ubota-{L}eopoldt {$L$}-functions.
\newblock {\em Proc. Edinb. Math. Soc. (2)}, 65(2):460--479, 2022.

\end{thebibliography}
\end{document}